\documentclass[11pt]{article}

\usepackage[a4paper,margin=1in]{geometry}
\usepackage{amsmath,amssymb,amsthm,mathtools,bm,mathrsfs}
\usepackage{microtype}
\usepackage{cite}
\usepackage[hidelinks]{hyperref}

\numberwithin{equation}{section}

\newtheorem{theorem}{Theorem}[section]
\newtheorem{lemma}[theorem]{Lemma}
\newtheorem{proposition}[theorem]{Proposition}
\newtheorem{corollary}[theorem]{Corollary}
\theoremstyle{definition}

\theoremstyle{remark}
\newtheorem{remark}[theorem]{Remark}

\newcommand{\R}{\mathbb{R}}
\newcommand{\C}{\mathbb{C}}
\newcommand{\eps}{\varepsilon}
\newcommand{\pa}{\partial}
\newcommand{\Om}{\Omega}
\newcommand{\esssupp}{\operatorname{ess\,supp}}
\newcommand{\cL}{\mathcal{L}}
\newcommand{\cA}{\mathcal{A}}
\newcommand{\cH}{\mathcal{H}}
\newcommand{\cQ}{\mathcal{Q}}
\newcommand{\cB}{\mathcal{B}}
\newcommand{\bfu}{\bm{u}}
\newcommand{\bfv}{\bm{v}}
\newcommand{\bfw}{\bm{w}}
\newcommand{\bff}{\bm{f}}
\newcommand{\bfg}{\bm{g}}
\newcommand{\bfF}{\bm{F}}
\newcommand{\bfG}{\bm{G}}
\newcommand{\bfH}{\bm{H}}
\newcommand{\bfh}{\bm{h}}
\newcommand{\bfM}{\bm{M}}
\newcommand{\bfC}{\bm{C}}
\newcommand{\bfU}{\bm{U}}
\newcommand{\bfV}{\bm{V}}
\newcommand{\bfW}{\bm{W}}
\newcommand{\bfa}{\bm{a}}
\newcommand{\bfb}{\bm{b}}
\newcommand{\bfGamma}{\bm{\Gamma}}
\newcommand{\bfK}{\bm{K}}
\newcommand{\bfR}{\bm{R}}
\newcommand{\bfZ}{\bm{Z}}
\newcommand{\bfE}{\bm{E}}
\newcommand{\bfnu}{\bm{\nu}}
\newcommand{\Tnu}{\bm{T}_{\nu}}

\title{Low-Frequency Passive Identification of Structured Elastic Density and Initial States}

\author{
Yixian Gao\thanks{School of Mathematics and Statistics, Center for Mathematics and Interdisciplinary Sciences,
Northeast Normal University, Changchun, Jilin 130024, China.
Email: \texttt{gaoyx643@nenu.edu.cn}.}
\and
Hongyu Liu\thanks{Department of Mathematics, City University of Hong Kong,
Kowloon, Hong Kong, China.
Email: \texttt{hongyu.liuip@gmail.com, hongyliu@cityu.edu.hk}.}
\and
Yang Liu\thanks{School of Mathematics and Statistics, Center for Mathematics and Interdisciplinary Sciences,
Northeast Normal University, Changchun, Jilin 130024, China.
Email: \texttt{liuy694@nenu.edu.cn}.}
}

\date{}

\begin{document}

\maketitle

\begin{abstract}
We study simultaneous identification of the mass density, initial displacement, and initial
velocity for the three-dimensional isotropic elastic wave equation with known constant
Lam\'e parameters.  The density equals a known positive constant outside an interior set.
The data are the complete displacement trace on an enclosing boundary.  We first assume
that the density-weighted initial displacement and velocity have fixed known profiles in one
spatial direction.  The $s^0$ and $s^1$ coefficients of the zero-frequency Laplace
expansion identify these weighted states.  The $s^2$ and $s^3$ coefficients then give
static Lam\'e orthogonality identities for the density difference.  The exact difference
expansion through order $s^3$ has an $O(s^4)$ remainder, uniformly on bounded spatial sets
and bounded density-contrast and weighted-state classes.  Under alignment of the two
initial states and a nonzero moment of the density-weighted initial velocity, the two
density identities reduce to a constant-vector static transform.  Two opposite elastic
null phases give uniqueness for one or two fixed vertical density profiles when the profile
family is two-sided Laplace nondegenerate.  When $\lambda+\mu\ne0$, each nonzero
normal root has partial multiplicities $2$ and $1$ and admits a length-two Jordan chain.
The associated polynomial--exponential Lam\'e mode produces derivatives of the bilateral
profile transforms.  The resulting Hermite--Laplace system gives uniqueness for aligned
classes with up to four fixed vertical profiles and independent horizontal coefficients,
provided that the profile system is Hermite--Laplace nondegenerate.
Distinct translations of one compactly supported profile provide an explicit four-profile
class.  We also prove rigidity of the alignment reduction and exhibit an infinite-dimensional
kernel for the reduced transform on unrestricted densities.
\end{abstract}



\section{Introduction}

\subsection{Problem and principal results}

Let $\Om\subset\R^3$ be a bounded Lipschitz domain and let $Q$ be a bounded open set with
$\overline Q\subset\Om$.  We consider the whole-space Cauchy problem
\begin{equation}
\label{eq:main-time}
 \rho(x)\pa_t^2\bfu(x,t)-\cL_{\lambda,\mu}\bfu(x,t)=\bm{0},
 \qquad (x,t)\in\R^3\times(0,\infty),
\end{equation}
where
\[
 \cL_{\lambda,\mu}\bfv
 :=\mu\Delta\bfv+(\lambda+\mu)\nabla(\nabla\cdot\bfv),
 \qquad
 \mu>0,
 \quad
 \lambda+2\mu>0,
\]
and
\begin{equation}
\label{eq:initial-data}
 \bfu(x,0)=\bff(x),
 \qquad
 \pa_t\bfu(x,0)=\bfg(x).
\end{equation}
The Lam\'e parameters are known constants.  The unknown density equals a known constant
$\rho_0>0$ outside $Q$, and the initial data are supported in $Q$.  The passive observation is
\[
 \mathscr M_{\rho,\bff,\bfg}
 :=\left.\bfu\right|_{\pa\Om\times(0,\infty)}.
\]
By the trace theorem and Proposition~\ref{prop:forward}, this is an element of
$C([0,\infty);H^{1/2}(\pa\Om)^3)$.
Because the exterior coefficients and exterior initial data are known, equality of these
traces determines equality of the exterior Laplace fields.  This standard exterior fact is
used only to compare the coefficients of the zero-frequency expansion.

Set $\cA_{\lambda,\mu}:=-\cL_{\lambda,\mu}$ and introduce
\[
 \bfF:=\rho\bff,
 \qquad
 \bfG:=\rho\bfg,
 \qquad
 q:=\rho-\rho_0.
\]
For $s>0$, let $\widetilde{\bfu}(\cdot,s)$ denote the Laplace transform of the
displacement.  It satisfies
\[
 (\cA_{\lambda,\mu}+\rho_0s^2)\widetilde{\bfu}
 =\bfG+s\bfF-s^2q\widetilde{\bfu}.
\]
Thus, after $(\bfF,\bfG)$ are treated as the first-stage unknowns, the independent density
contrast $q$ first enters explicitly through $-s^2q\widetilde{\bfu}$.

Assume first that $Q=D\times I$, where $D\subset\R^2$ is bounded and open and $I$ is a
bounded interval.  Write $x=(x',x_3)$ with $x'\in\R^2$.  Suppose that the weighted
initial states have fixed known profiles in the distinguished direction,
\[
 \rho_j\bff_j=\bfF_j^{\flat}(x')\psi_f(x_3),
 \qquad
 \rho_j\bfg_j=\bfG_j^{\flat}(x')\psi_g(x_3).
\]
Theorem~\ref{thm:weighted} shows that one passive boundary history uniquely determines
$\rho\bff$ and $\rho\bfg$.  The proof uses the vector structure of the Lam\'e system.
For a horizontal frequency $\eta\ne\bm{0}$, one elastic null phase leaves a residual
polarization line.  The two normal phases with exponents $\pm|\eta|$ remove it.

After the weighted states have been identified, define
\[
 \bfU_0:=\bfGamma_0*\bfG,
 \qquad
 \bfU_1:=\bfGamma_0*\bfF-\gamma_1\int_Q\bfG(y)\,dy,
\]
where $\bfGamma_0$ is the Kelvin tensor and $\gamma_1>0$ is the coefficient of the linear
term in the modified Kupradze expansion.  Equal passive data then imply
\begin{equation}
\label{eq:intro-density-identities}
 \int_Q(\rho_1-\rho_2)\bfU_0\cdot\bfv\,dx=0,
 \qquad
 \int_Q(\rho_1-\rho_2)\bfU_1\cdot\bfv\,dx=0,
 \qquad \bfv\in\cH_{\lambda,\mu}.
\end{equation}
Here $\cH_{\lambda,\mu}$ denotes the space of globally defined static Lam\'e fields.
The second identity is derived from the exact difference of the two
Lippmann--Schwinger equations through order $s^3$.  The remainder is
$O(s^4)$ in $L^\infty$ on bounded sets.  The vector moment at order $s^3$ vanishes by the
preceding identity.

The identities \eqref{eq:intro-density-identities} give two types of density result.  In a
finite-dimensional family, Theorem~\ref{thm:finite-rank} gives a real-rank criterion that
can be checked after the weighted state has been identified.  For an infinite-dimensional
multi-profile class, we assume the algebraic
alignment
\begin{equation}
\label{eq:intro-aligned-class}
 \bff=\beta\bfg,
 \qquad
 \int_Q\rho\bfg\,dx\ne\bm{0}.
\end{equation}
Here $\beta\in\R$.  Then $\bfU_1-\beta\bfU_0$ is a nonzero constant vector.  For one or
two prescribed vertical density profiles, the two pure normal phases provide enough scalar
equations when the profiles are two-sided Laplace nondegenerate.  For three
or four profiles, an additional normal-mode structure is used.  Fixing $\eta\ne\bm{0}$ and writing
\[
 \zeta(r)=(i\eta_1,i\eta_2,r),
\]
the static Lam\'e equation has the normal matrix polynomial
\[
 \mathscr P_\eta(r)
 =\mu(r^2-|\eta|^2)I_3+(\lambda+\mu)\zeta(r)\zeta(r)^\top.
\]
When $\lambda+\mu\ne0$, each root $r=\pm|\eta|$ has algebraic multiplicity three and
geometric multiplicity two and admits a length-two Jordan chain.  Its associated
polynomial--exponential Lam\'e mode produces the first derivative of the bilateral profile
transform at that root.  The two roots therefore provide the transform values and first
derivatives at both signs, leading to the two-sided Hermite--Laplace matrix.
Theorem~\ref{thm:global} proves injectivity for Hermite--Laplace nondegenerate aligned classes
containing up to four fixed vertical profiles.  Distinct translations of one compactly
supported profile give an explicit nondegenerate four-profile family.

The use of a Jordan chain here follows the standard spectral terminology for regular matrix
polynomials \cite{GohbergLancasterRodman2009}.  Polynomial--exponential solutions associated
with such chains are classical at the level of constant-coefficient systems.  The new point
in this paper is their explicit realization for the static Lam\'e normal polynomial and their
use to generate the derivative rows needed for passive multi-profile identifiability.

The alignment in \eqref{eq:intro-aligned-class} is structural and is not claimed to be
generic.  Proposition~\ref{prop:alignment-rigidity} shows that replacing the resulting
constant vector by a factor $\bfM\chi(x_3)$ on the whole cylinder gives no larger class.
Compact support forces the exact alignment again.
Proposition~\ref{prop:general-density-obstruction} exhibits an infinite-dimensional kernel for the
reduced constant-vector transform on unrestricted densities.  Neither result rules out
larger identifiable classes based on the full pair $(\bfU_0,\bfU_1)$.  The present results
are qualitative and use the complete time history; finite-time identification and stability
from noisy traces remain open.

\subsection{Relation to previous work}

Inverse coefficient problems for the nonstationary Lam\'e system have been studied through
Carleman estimates, dynamical boundary maps, high-frequency constructions, and nonlinear
interaction methods; see
\cite{ImanuvilovYamamoto2005,BellassouedImanuvilovYamamoto2008,Rachele2003,
DeHoopNakamuraZhai2017,Bhattacharyya2018,DeHoopNakamuraZhai2019,
BhattacharyyaDeHoopKatsnelsonUhlmann2022SIIMS,UhlmannZhai2024,Zhai2026}.
Inverse elastic sources and nonradiating sources are treated in
\cite{Tittelfitz2012,BaoHuKianYin2018,HuKian2020,BlastenLin2019}.

For scalar waves, simultaneous passive source--medium identification includes
\cite{StefanovUhlmann2013,LiuUhlmann2015,KnoxMoradifam2020,Feizmohammadi2025,KianUhlmann2025}.
Kian and Liu \cite{KianLiu2025} prove uniqueness and H\"older stability from a single passive
boundary measurement for broad piecewise-constant sound speeds, without a temporal-decay
hypothesis.  Moradifam \cite{Moradifam2026} obtains global uniqueness and Lipschitz
stability under a constitutive source--speed relation and geometric visibility assumptions.
The evolutionary framework of Chen, Jiang, Liu, Lo, and Tao
\cite{ChenJiangLiuLoTao2026} combines low- and high-frequency information and treats
several jointly unknown quantities under structural hypotheses.  The absence of a decay
assumption is therefore not the distinguishing claim of the present paper.

The plate problem in \cite{GaoLiuLiu2023} also uses passive data and low-frequency
asymptotics.  Three differences matter in elasticity.  First, the background resolvent is matrix valued.
Second, static exponential amplitudes satisfy the polarization constraint
$\zeta\cdot\bfa=0$.  Third, one null phase leaves a one-dimensional polarization line.
The opposite phase removes this line.  The nonsemisimple normal root also produces a
polynomial--exponential mode and hence a transform derivative.  We combine these elastic features with the exact order-$s^3$ difference identity.  This
yields the passive-identification results proved below.

Section~\ref{sec:setting} gives the functional setting and states the principal results.
Section~\ref{sec:static} develops the static potential theory, the normal matrix polynomial,
and the pure and generalized completeness arguments.  Section~\ref{sec:laplace} derives
the transformed equation and identifies the weighted initial state.
Section~\ref{sec:density} proves the two density identities and their uniqueness consequences.
Section~\ref{sec:scope} records the limitations and open problems.  Technical kernel and
resolvent arguments are placed in the appendices.

\section{Setting and principal statements}
\label{sec:setting}

\subsection{Notation and the whole-space forward problem}

We write $\R_+=(0,\infty)$, $i=\sqrt{-1}$,
$x=(x',x_3)$ with $x'=(x_1,x_2)$, $e_3=(0,0,1)^\top$, and $I_3$ for the
$3\times3$ identity matrix.  For complex vectors,
$\bfa\cdot\bfb=\sum_{k=1}^3a_kb_k$ denotes the complex bilinear Euclidean product;
conjugation is written explicitly in sesquilinear forms.  For matrices,
$A:B=\sum_{j,k}A_{jk}B_{jk}$, and $|\cdot|$ denotes the Euclidean or Frobenius norm,
as appropriate.  The Fourier transform is
\begin{equation*}
 \widehat h(\xi):=\int_{\R^d}e^{-ix\cdot\xi}h(x)\,dx,
 \qquad
 h\in L^1(\R^d),
\end{equation*}
with inverse factor $(2\pi)^{-d}$.  All inequalities between measurable coefficients hold
almost everywhere.

For $\bfv,\bfw\in C_c^\infty(\R^3)^3$, set
\begin{equation*}
 a(\bfv,\bfw)
 :=\int_{\R^3}
 \left[
 \lambda(\nabla\cdot\bfv)(\nabla\cdot\overline{\bfw})
 +2\mu\eps(\bfv):\eps(\overline{\bfw})
 \right]dx,
\end{equation*}
where $\eps(\bfv)=\tfrac12(\nabla\bfv+(\nabla\bfv)^\top)$.  The form extends
continuously to $H^1(\R^3)^3$.  With
\[
 P_L(\xi)=\frac{\xi\xi^\top}{|\xi|^2},
 \qquad
 P_T(\xi)=I_3-P_L(\xi),
 \qquad \xi\ne0,
\]
Plancherel's theorem gives
\begin{equation}
\label{eq:form-Fourier}
 a(\bfv,\bfv)
 =(2\pi)^{-3}\int_{\R^3}
 \left[
 \mu|\xi|^2|P_T(\xi)\widehat{\bfv}(\xi)|^2
 +(\lambda+2\mu)|\xi|^2|P_L(\xi)\widehat{\bfv}(\xi)|^2
 \right]d\xi.
\end{equation}

In particular, there are constants $0<c_{\mathrm{el}}\le C_{\mathrm{el}}$ such that
\begin{equation*}
 c_{\mathrm{el}}\|\nabla\bfv\|_{L^2}^2
 \le a(\bfv,\bfv)
 \le C_{\mathrm{el}}\|\nabla\bfv\|_{L^2}^2,
 \qquad \bfv\in H^1(\R^3)^3.
\end{equation*}

\begin{proposition}
\label{prop:forward}
Let $\rho\in L^\infty(\R^3)$ satisfy
$0<\rho_-\le\rho\le\rho_+<\infty$.  For
$\bff\in H^1(\R^3)^3$ and $\bfg\in L^2(\R^3)^3$, the Cauchy problem
\eqref{eq:main-time}--\eqref{eq:initial-data} has a unique solution
\[
 \bfu\in C([0,\infty);H^1(\R^3)^3)
       \cap C^1([0,\infty);L^2(\R^3)^3).
\]
The quadratic energy
\begin{equation}
\label{eq:elastic-energy}
 E_{\bfu}(t)
 :=\frac12\int_{\R^3}
 \left[
 \rho|\pa_t\bfu|^2
 +\lambda|\nabla\cdot\bfu|^2
 +2\mu|\eps(\bfu)|^2
 \right]dx
\end{equation}
is independent of $t$ and nonnegative.  Moreover,
\begin{equation}
\label{eq:H1-growth}
 \|\bfu(t)\|_{H^1(\R^3)}
 +\|\pa_t\bfu(t)\|_{L^2(\R^3)}
 \le C(1+t),
 \qquad t\ge0,
\end{equation}
where $C$ depends only on $\lambda,\mu,\rho_-,\rho_+$ and the norms of the initial data.
Consequently the $H^1$-valued Laplace transform exists for every $s>0$.
\end{proposition}

\begin{proof}
Let $H_\rho=L^2(\R^3,\rho\,dx)^3$ and $V=H^1(\R^3)^3$.  By
\eqref{eq:form-Fourier}, the form $a$ is Hermitian, nonnegative, and closed on $V$ after
addition of the $H_\rho$ norm.  Let $\mathbb A_\rho$ be the associated nonnegative
self-adjoint operator in $H_\rho$.  Weakly,
$\mathbb A_\rho=\rho^{-1}\cA_{\lambda,\mu}$ and
$D(\mathbb A_\rho^{1/2})=V$.  Define
\[
 m_t(\varsigma)=
 \begin{cases}
 \dfrac{\sin(t\sqrt\varsigma)}{\sqrt\varsigma},&\varsigma>0,\\[1mm]
 t,&\varsigma=0.
 \end{cases}
\]
The spectral theorem gives
\[
 \bfu(t)=\cos(t\mathbb A_\rho^{1/2})\bff+m_t(\mathbb A_\rho)\bfg
\]
and the asserted energy solution.  The conserved quantity is
$\|\pa_t\bfu(t)\|_{H_\rho}^2+a(\bfu(t),\bfu(t))=2E_{\bfu}(t)$; nonnegativity follows from
\eqref{eq:form-Fourier}.  Energy conservation controls
$\|\pa_t\bfu(t)\|_{L^2}$ and $\|\nabla\bfu(t)\|_{L^2}$.  Since
\[
 \bfu(t)=\bff+\int_0^t\pa_\upsilon\bfu(\upsilon)\,d\upsilon,
\]
the $L^2$ norm grows at most linearly, which proves \eqref{eq:H1-growth}.  The Laplace
integral converges because $e^{-st}(1+t)$ is integrable.
\end{proof}

\begin{remark}
The nonnegativity of \eqref{eq:elastic-energy} under
$\mu>0$ and $\lambda+2\mu>0$ is a property of the integrated whole-space quadratic form,
as seen from \eqref{eq:form-Fourier}.  We do not use pointwise positive definiteness of the
elastic energy density, which would require $3\lambda+2\mu>0$.
\end{remark}

Fix constants $0<\rho_-<\rho_0<\rho_+$.  Throughout the inverse problem, for
$j=1,2$, we assume
\begin{equation}
\label{eq:basic-admissibility}
 \rho_j\in L^\infty(\R^3;\R),
 \quad
 \rho_-\le\rho_j\le\rho_+,
 \quad
 \rho_j=\rho_0\ \text{in }\R^3\setminus Q,
\end{equation}
and
\begin{equation}
\label{eq:basic-initial-regularity}
 \bff_j\in H^1_c(Q;\R^3),
 \qquad
 \bfg_j\in L^2_c(Q;\R^3).
\end{equation}
For an open set $O\subset\R^d$, the notation $X_c(O)$ denotes the elements of
$X(\R^d)$ whose essential support is compactly contained in $O$.

We shall use the space of globally defined static Lam\'e fields
\begin{equation*}
 \cH_{\lambda,\mu}
 :=\{\bfv\in H^1_{\mathrm{loc}}(\R^3;\C^3):
       \cL_{\lambda,\mu}\bfv=0\text{ in }\R^3\}.
\end{equation*}

\subsection{Identification of the weighted initial state}

Assume from now on that
\begin{equation}
\label{eq:cylinder}
 Q=D\times I,
 \qquad
 \overline{D\times I}\subset\Om,
\end{equation}
where $D\subset\R^2$ is bounded and open and $I=(t_-,t_+)$.  For
$\psi\in C_c^\infty(I)$, define the bilateral Laplace transform
\begin{equation*}
 \mathscr L_\psi(z):=\int_I\psi(t)e^{zt}\,dt,
 \qquad z\in\C.
\end{equation*}
If $\psi\not\equiv0$, then $\mathscr L_\psi$ is a nontrivial entire function.  Hence
\begin{equation*}
 \mathcal S_\psi
 :=\{\tau>0:\mathscr L_\psi(\tau)\mathscr L_\psi(-\tau)\ne0\}
\end{equation*}
is open and dense in $(0,\infty)$.

Fix known nonzero profiles $\psi_f,\psi_g\in C_c^\infty(I;\R)$.  We assume
\begin{equation}
\label{eq:weighted-structure}
 \rho_j\bff_j=\bfF_j^{\flat}(x')\psi_f(x_3),
 \qquad
 \rho_j\bfg_j=\bfG_j^{\flat}(x')\psi_g(x_3),
\end{equation}
with $\bfF_j^{\flat},\bfG_j^{\flat}\in L^2_c(D;\R^3)$.

\begin{theorem}
\label{thm:weighted}
Assume \eqref{eq:basic-admissibility}, \eqref{eq:basic-initial-regularity},
\eqref{eq:cylinder}, and \eqref{eq:weighted-structure}.  If
\begin{equation}
\label{eq:data-equality}
 \mathscr M_{\rho_1,\bff_1,\bfg_1}
 =\mathscr M_{\rho_2,\bff_2,\bfg_2},
\end{equation}
then
\[
 \rho_1\bff_1=\rho_2\bff_2,
 \qquad
 \rho_1\bfg_1=\rho_2\bfg_2
 \quad\text{a.e. in }\R^3.
\]
\end{theorem}

For any pair satisfying the assumptions of Theorem~\ref{thm:weighted} and
\eqref{eq:data-equality}, we write the common weighted states as
\begin{equation*}
 \bfF:=\rho_1\bff_1=\rho_2\bff_2,
 \qquad
 \bfG:=\rho_1\bfg_1=\rho_2\bfg_2.
\end{equation*}
Let $\bfGamma_0$ denote the Kelvin tensor and set
\begin{equation}
\label{eq:common-U01}
 \bfU_0:=\bfGamma_0*\bfG,
 \qquad
 \bfU_1:=\bfGamma_0*\bfF-\gamma_1\int_Q\bfG(y)\,dy,
\end{equation}
where
\begin{equation}
\label{eq:gamma1}
 \gamma_1
 =\frac{\sqrt{\rho_0}}{12\pi}
 \left[2\mu^{-3/2}+(\lambda+2\mu)^{-3/2}\right]>0.
\end{equation}

Propositions~\ref{prop:s2-density} and \ref{prop:s3-density} prove the common reduction
used by all density results below: with $\delta\rho:=\rho_1-\rho_2$,
\begin{equation*}
 \int_Q\delta\rho(x)\bfU_m(x)\cdot\bfv(x)\,dx=0,
 \qquad m=0,1,
 \qquad \bfv\in\cH_{\lambda,\mu}.
\end{equation*}

\subsection{Finite-dimensional density families}

Let $\phi_1,\ldots,\phi_M\in L^\infty_c(Q;\R)$ be linearly independent, let
$\mathfrak P_M\subset\R^M$ be a prescribed nonempty parameter set, and put
\[
 \rho_\alpha:=\rho_0+\sum_{m=1}^M\alpha_m\phi_m,
 \qquad
 \cQ_M
 :=\left\{
 \rho_\alpha:
 \alpha\in\mathfrak P_M,
 \quad \rho_-\le\rho_\alpha\le\rho_+\ \text{a.e. in }\R^3
 \right\}.
\]
For static probes $\bfv_1,\ldots,\bfv_L\in\cH_{\lambda,\mu}$ and
$1\le\ell\le L$, $1\le m\le M$, set
\[
 \mathbb M^{(0)}_{\ell m}
 :=\int_Q\phi_m(x)\bfU_0(x)\cdot\bfv_\ell(x)\,dx,
 \qquad
 \mathbb M^{(1)}_{\ell m}
 :=\int_Q\phi_m(x)\bfU_1(x)\cdot\bfv_\ell(x)\,dx.
\]
Form the realified matrix
\begin{equation*}
 \mathbb M_{\R}
 :=\begin{pmatrix}
 \Re\mathbb M^{(0)}\\
 \Im\mathbb M^{(0)}\\
 \Re\mathbb M^{(1)}\\
 \Im\mathbb M^{(1)}
 \end{pmatrix}
 \in\R^{4L\times M}.
\end{equation*}

\begin{theorem}
\label{thm:finite-rank}
Assume the hypotheses of Theorem~\ref{thm:weighted}, the data equality
\eqref{eq:data-equality}, and $\rho_1,\rho_2\in\cQ_M$.  Form
$\bfU_0,\bfU_1$ and $\mathbb M_{\R}$ from the common weighted states identified by that
theorem.  If
\begin{equation}
\label{eq:real-full-rank}
 \operatorname{rank}_{\R}\mathbb M_{\R}=M,
\end{equation}
then
\[
 \rho_1=\rho_2,
 \qquad
 \bff_1=\bff_2,
 \qquad
 \bfg_1=\bfg_2
 \quad\text{a.e. in }\R^3.
\]
\end{theorem}

The matrix depends on the weighted state identified from the data and is therefore an
\emph{a posteriori} certificate on that source fiber.  If the full family of static probes
separates the $M$-dimensional density space, elementary finite-dimensional duality selects
finitely many probes satisfying \eqref{eq:real-full-rank}.  The smallest singular value of
$\mathbb M_{\R}$ controls only the resulting coefficient-level linear system; no stability
claim from noisy time-domain data is made.

\subsection{Aligned multi-profile density classes}

Let $1\le N\le4$ and let
$\Psi=(\psi_1,\ldots,\psi_N)$ be fixed real-valued profiles in
$C_c^\infty(I)$.  For $N\le2$, define the two-sided value matrix
\begin{equation*}
 \mathscr V_{\Psi}(z)
 :=
 \begin{pmatrix}
  \mathscr L_{\psi_1}(z)&\cdots&\mathscr L_{\psi_N}(z)\\
  \mathscr L_{\psi_1}(-z)&\cdots&\mathscr L_{\psi_N}(-z)
 \end{pmatrix}.
\end{equation*}
We call $\Psi$ \emph{two-sided Laplace nondegenerate} if some $N\times N$ minor of
$\mathscr V_\Psi$ is not identically zero.  For $N\le4$, define the two-sided first-order
Hermite--Laplace matrix
\begin{equation*}
 \mathscr H_{\Psi}(z)
 :=
 \begin{pmatrix}
  \mathscr L_{\psi_1}(z)&\cdots&\mathscr L_{\psi_N}(z)\\
  \mathscr L'_{\psi_1}(z)&\cdots&\mathscr L'_{\psi_N}(z)\\
  \mathscr L_{\psi_1}(-z)&\cdots&\mathscr L_{\psi_N}(-z)\\
  \mathscr L'_{\psi_1}(-z)&\cdots&\mathscr L'_{\psi_N}(-z)
 \end{pmatrix}.
\end{equation*}
Here
\[
 \mathscr L'_{\psi}(z)=\int_I t\psi(t)e^{zt}\,dt,
\]
and $\mathscr L'_{\psi}(-z)$ denotes the derivative evaluated at $-z$.  We call $\Psi$
\emph{Hermite--Laplace nondegenerate} if some $N\times N$ minor of
$\mathscr H_\Psi$ is not identically zero.  Equivalently, the corresponding matrix has full
column rank at some point of $\C$.  Since its minors are entire, either
nondegeneracy condition implies full column rank on an open dense subset of $(0,\infty)$.
We write these full-rank sets as $\mathcal S_\Psi^{(0)}$ and
$\mathcal S_\Psi^{(1)}$, respectively.

If $J$ is a nonempty compact subset of $\mathcal S_\Psi^{(1)}$, continuity gives
\begin{equation}
\label{eq:Hermite-uniform-gap}
 \inf_{\tau\in J}\sigma_{\min}(\mathscr H_\Psi(\tau))>0.
\end{equation}
This controls only the algebraic profile matrix after the static equations have been
normalized.  Uniform conditioning of the full probing system also requires bounded
normalizations of the pure-mode amplitudes and lower bounds for the derivative-row
polarization factors.  More precisely, suppose $0<\tau_0<\tau_1$ and
$[\tau_0,\tau_1]\subset J$.  Fix
$\bfM=(\bfM',M_3)\in\R^3\setminus\{\bm{0}\}$ and
$\eta\in\R^2\setminus\{\bm{0}\}$, and set
$\tau=|\eta|$ and $\zeta_\pm=(i\eta_1,i\eta_2,\pm\tau)$.  On
$\tau\in[\tau_0,\tau_1]$ one has
$|\bfM\cdot\zeta_\pm|\ge |M_3|\tau_0$ when $M_3\ne0$; if $M_3=0$, the angular variable
must remain a positive distance from the line $\bfM'\cdot\eta=0$.  No stability estimate
from noisy time-domain data is asserted.

Let $\mathfrak A_0$ denote the class of individual triples $(\rho,\bff,\bfg)$ satisfying
\begin{equation*}
 \rho\in L^\infty(\R^3;\R),
 \quad \rho_-\le\rho\le\rho_+,
 \quad \rho=\rho_0\ \text{in }\R^3\setminus Q,
 \quad \bff\in H^1_c(Q;\R^3),
 \quad \bfg\in L^2_c(Q;\R^3).
\end{equation*}
Fix a known source profile $\psi_{\mathrm{src}}\in C_c^\infty(I;\R)$ with
$\int_I\psi_{\mathrm{src}}(t)\,dt\ne0$.  Fix density profiles
$\Psi_\rho=(\psi_{\rho,1},\ldots,\psi_{\rho,N})\in(C_c^\infty(I;\R))^N$; these profiles
remain fixed throughout.  Let $\mathfrak A^{(N)}$ consist of the triples in
$\mathfrak A_0$ for which
\begin{equation}
\label{eq:N-profile-density}
 \rho(x)=\rho_0+\sum_{k=1}^Np_k(x')\psi_{\rho,k}(x_3),
 \qquad
 p_k\in W^{1,\infty}_c(D;\R),
\end{equation}
and
\begin{equation}
\label{eq:aligned-weighted}
 \rho\bfg=\bfG^{\flat}(x')\psi_{\mathrm{src}}(x_3),
 \qquad
 \rho\bff=\beta\bfG^{\flat}(x')\psi_{\mathrm{src}}(x_3),
\end{equation}
for some $\bfG^{\flat}\in H^1_c(D;\R^3)$ and $\beta\in\R$, with
\begin{equation*}
 \int_Q\rho\bfg\,dx\ne\bm{0}.
\end{equation*}
The regularity in \eqref{eq:N-profile-density} and the positive lower bound on $\rho$
allow division by $\rho$.  Together with $\bfG^{\flat}\in H^1_c(D)$, this preserves the
required $H^1$ regularity of the aligned initial displacement.  Since $\rho>0$,
\eqref{eq:aligned-weighted} is equivalent to $\bff=\beta\bfg$.  The class is nonempty
and infinite dimensional.

\begin{theorem}
\label{thm:global}
Let $1\le N\le4$ and consider two triples in $\mathfrak A^{(N)}$ with the same fixed
source and density profiles.
Assume either
\begin{enumerate}
\item[(i)] $1\le N\le2$ and $\Psi_\rho$ is two-sided Laplace nondegenerate; or
\item[(ii)] $1\le N\le4$, $\lambda+\mu\ne0$, and $\Psi_\rho$ is
Hermite--Laplace nondegenerate.
\end{enumerate}
Then equality of the passive displacement histories implies
\[
 \rho_1=\rho_2,
 \qquad
 \bff_1=\bff_2,
 \qquad
 \bfg_1=\bfg_2
 \quad\text{a.e. in }\R^3.
\]
Equivalently, the passive map is injective on $\mathfrak A^{(N)}$ under either condition.
\end{theorem}

Proposition~\ref{prop:translated-Hermite} supplies an explicit nondegenerate family.

\begin{proposition}
\label{prop:translated-Hermite}
Let $\varphi\in C_c^\infty(\R;\R)$ be nonzero and choose
$d_1<d_2<d_3<d_4$ so that
$\operatorname{supp}\varphi+d_k\Subset I$.  Set
\[
 \psi_k(t):=\varphi(t-d_k),
 \qquad k=1,2,3,4.
\]
Then the four-profile family is Hermite--Laplace nondegenerate.  Every one- or two-profile
subfamily is two-sided Laplace nondegenerate, and every subfamily of at most four profiles is
Hermite--Laplace nondegenerate.
\end{proposition}

Thus, in Theorem~\ref{thm:global}, the choice
$\psi_{\rho,k}=\psi_k$ gives an explicit nondegenerate density-profile family.

The alignment is a structural assumption.  We do not claim that it is generic.  The
next result identifies a rigidity of the constant-vector reduction used in the proof.

\begin{proposition}
\label{prop:alignment-rigidity}
Let $\bfF,\bfG\in L^2_c(Q;\R^3)$ and define $\bfU_0,\bfU_1$ by
\eqref{eq:common-U01}.  Suppose that for some $c\in\R$, some
$\bfM\in\R^3\setminus\{\bm{0}\}$, and some scalar distribution $\chi\in\mathcal D'(I)$,
\[
 \bfU_1(x)-c\bfU_0(x)=\bfM\chi(x_3)
 \qquad\text{in }\mathcal D'(Q)^3.
\]
Then
\begin{equation}
\label{eq:rigidity-alignment}
 \bfF=c\bfG\quad\text{a.e. in }Q,
\end{equation}
and
\begin{equation}
\label{eq:rigidity-constant}
 \bfU_1-c\bfU_0=-\gamma_1\int_Q\bfG(y)\,dy
 \qquad\text{in }Q.
\end{equation}
Thus the one-directional factorization on the whole cylinder reduces to the aligned
constant-vector case.  If $\bfM\chi\not\equiv\bm{0}$ in $Q$, then
$\int_Q\bfG(y)\,dy\ne\bm{0}$.
\end{proposition}

\subsection{A conditional monotonicity criterion}

\begin{corollary}
\label{cor:monotonicity}
Assume the hypotheses of Theorem~\ref{thm:weighted} and the data equality
\eqref{eq:data-equality}.  Let $\bfU_0$ be the common field in
\eqref{eq:common-U01}.  Suppose that there is a real field
$\bfv_*\in\cH_{\lambda,\mu}$ such that
\begin{equation}
\label{eq:positive-probe}
 \bfU_0(x)\cdot\bfv_*(x)>0
 \quad\text{for a.e. }x\in Q.
\end{equation}
If either $\rho_1\le\rho_2$ or $\rho_2\le\rho_1$ a.e. in $Q$, then
\[
 \rho_1=\rho_2,
 \qquad
 \bff_1=\bff_2,
 \qquad
 \bfg_1=\bfg_2
 \quad\text{a.e. in }\R^3.
\]
\end{corollary}

\section{Static Lam\'e potentials and completeness}
\label{sec:static}

\subsection{Complex static exponentials}

Every constant vector belongs to $\cH_{\lambda,\mu}$.  More generally, if
$\zeta,\bfa\in\C^3$ satisfy
\begin{equation}
\label{eq:null-conditions}
 \zeta\cdot\zeta=0,
 \qquad
 \zeta\cdot\bfa=0,
\end{equation}
then
\[
 \bfv_{\zeta,\bfa}(x):=\bfa e^{\zeta\cdot x}
\]
belongs to $\cH_{\lambda,\mu}$, because
\[
 \cL_{\lambda,\mu}\bfv_{\zeta,\bfa}
 =\left[
 \mu(\zeta\cdot\zeta)\bfa
 +(\lambda+\mu)\zeta(\zeta\cdot\bfa)
 \right]e^{\zeta\cdot x}=\bm{0}.
\]
The bilinear complexification in \eqref{eq:null-conditions} is essential: nonzero null
vectors do not exist for the Hermitian product.

\subsection{Kelvin potentials and static orthogonality}

The Kelvin tensor of $\cA_{\lambda,\mu}$ is
\begin{equation}
\label{eq:kelvin}
 \bfGamma_0(x)
 =\frac{\lambda+3\mu}{8\pi\mu(\lambda+2\mu)}\frac{I_3}{|x|}
 +\frac{\lambda+\mu}{8\pi\mu(\lambda+2\mu)}\frac{xx^\top}{|x|^3},
 \qquad x\ne0,
\end{equation}
and $\cA_{\lambda,\mu}\bfGamma_0=\delta_0I_3$ in distributions.  Let $\bfnu$ denote
the outward unit normal to $\pa\Om$.  The traction is
\[
 \Tnu\bfv
 :=\left[\lambda(\nabla\cdot\bfv)I_3+2\mu\eps(\bfv)\right]\bfnu.
\]
For fields with the regularity used below, the bilinear Green identity (see, for example,
\cite{McLean2000}) is
\begin{equation*}
 \int_{\Om}
 \left[(\cA_{\lambda,\mu}\bfv)\cdot\bfw
       -\bfv\cdot(\cA_{\lambda,\mu}\bfw)\right]dx
 =\int_{\pa\Om}
 \left[\bfv\cdot\Tnu\bfw-(\Tnu\bfv)\cdot\bfw\right]dS.
\end{equation*}

\begin{lemma}
\label{lem:potential-orthogonality}
Let $\bfh\in L^2(\R^3)^3$ with $\esssupp\bfh\subset\overline Q$, and put
$\bfV=\bfGamma_0*\bfh$.  If
\[
 \bfV=\bm{0}\quad\text{in }\R^3\setminus\overline\Om,
\]
then
\begin{equation*}
 \int_Q\bfh(x)\cdot\bfv(x)\,dx=0
 \qquad\forall\bfv\in\cH_{\lambda,\mu}.
\end{equation*}
\end{lemma}

\begin{proof}
The whole-space potential satisfies
$\cA_{\lambda,\mu}\bfV=\bfh$ in distributions.  Since
$d_Q:=\operatorname{dist}(\overline Q,\pa\Om)>0$, the source vanishes in a full collar of
$\pa\Om$.  Constant-coefficient elliptic regularity in a smaller collar gives
$\bfV\in H^2_{\mathrm{loc}}(\R^3)^3$ and $\bfV$ is smooth in a full neighborhood of
$\pa\Om$.  The same smooth field vanishes on the exterior part of that neighborhood;
hence its displacement and traction traces vanish on $\pa\Om$.  For every $\bfv\in\cH_{\lambda,\mu}$, constant-coefficient elliptic regularity gives
$\bfv\in C^\infty(\R^3)^3$.  Since $\bfV$ is smooth in a neighborhood of the boundary,
the Green--Betti identity is legitimate for the pair $(\bfV,\bfv)$ (equivalently one may
use the weak traction identity and density).  Therefore
\begin{align*}
 \int_Q\bfh\cdot\bfv\,dx
 &=\int_{\Om}(\cA_{\lambda,\mu}\bfV)\cdot\bfv\,dx\\
 &=\int_{\Om}\bfV\cdot(\cA_{\lambda,\mu}\bfv)\,dx
   +\int_{\pa\Om}
    \left[\bfV\cdot\Tnu\bfv-(\Tnu\bfV)\cdot\bfv\right]dS\\
 &=0,
\end{align*}
because $\cA_{\lambda,\mu}\bfv=0$ and both boundary traces of $\bfV$ vanish.
\end{proof}

\subsection{One-profile vector completeness}

\begin{lemma}
\label{lem:completeness}
Let
\[
 \bfH(x)=\bfH_0(x')\psi(x_3),
 \qquad
 \bfH_0\in L^2_c(D)^3,
 \quad
 \psi\in C_c^\infty(I)\setminus\{0\}.
\]
If
\begin{equation}
\label{eq:H-all-orthogonal}
 \int_{\R^3}\bfH(x)\cdot\bfv(x)\,dx=0
 \qquad\forall\bfv\in\cH_{\lambda,\mu},
\end{equation}
then $\bfH_0=\bm{0}$ a.e. in $\R^2$.
\end{lemma}

\begin{proof}
Fix $\eta=(\eta_1,\eta_2)\in\R^2\setminus\{\bm{0}\}$ and define
\begin{equation*}
 \zeta_+=(i\eta_1,i\eta_2,|\eta|),
 \qquad
 \zeta_-=(i\eta_1,i\eta_2,-|\eta|).
\end{equation*}
Both vectors are null.  For arbitrary
$\bfa_\pm\in\zeta_\pm^\perp$, the fields
$\bfa_\pm e^{\zeta_\pm\cdot x}$ are static Lam\'e solutions.  Substitution into
\eqref{eq:H-all-orthogonal} yields
\begin{equation*}
 \widehat{\bfH_0}(-\eta)\cdot\bfa_\pm\,
 \mathscr L_\psi(\pm|\eta|)=0.
\end{equation*}
If $|\eta|\in\mathcal S_\psi$, both profile factors are nonzero.  Hence
$\widehat{\bfH_0}(-\eta)$ is orthogonal to $\zeta_\pm^\perp$ for both signs, and therefore
the nondegeneracy of the complex bilinear product gives
\begin{equation}
\label{eq:polarization-intersection}
 \widehat{\bfH_0}(-\eta)
 \in\operatorname{span}\{\zeta_+\}
   \cap\operatorname{span}\{\zeta_-\}
 =\{\bm{0}\}.
\end{equation}
Since $\mathcal S_\psi$ is dense in $(0,\infty)$, the set
$\{\eta\in\R^2\setminus\{\bm{0}\}:|\eta|\in\mathcal S_\psi\}$ is dense in
$\R^2\setminus\{\bm{0}\}$.  Since
$\bfH_0$ is compactly supported and belongs to $L^2$, its Fourier transform is continuous;
thus it vanishes everywhere.  Fourier injectivity proves the claim.
\end{proof}

The intersection in \eqref{eq:polarization-intersection} is the elasticity-specific step.
A single null phase leaves one residual polarization line; the opposite phase removes it.

\subsection{The normal matrix polynomial and multi-profile completeness}

Fix $\eta=(\eta_1,\eta_2)\in\R^2\setminus\{\bm{0}\}$, put $\tau=|\eta|$, and define
\[
 \zeta(r):=(i\eta_1,i\eta_2,r),
 \qquad
 \mathscr P_\eta(r)
 :=\mu(r^2-\tau^2)I_3+(\lambda+\mu)\zeta(r)\zeta(r)^\top.
\]
The terminology concerning eigenvectors and Jordan chains of regular matrix polynomials is
used in the standard sense of \cite{GohbergLancasterRodman2009}.

\begin{lemma}
\label{lem:normal-pencil}
For $\bfa\in\C^3$,
\begin{equation}
\label{eq:normal-symbol-action}
 \cL_{\lambda,\mu}
 \left(\bfa e^{i\eta\cdot x'+rx_3}\right)
 =\mathscr P_\eta(r)\bfa\,e^{i\eta\cdot x'+rx_3}.
\end{equation}
Moreover,
\begin{equation}
\label{eq:normal-determinant}
 \det\mathscr P_\eta(r)
 =\mu^2(\lambda+2\mu)(r^2-\tau^2)^3.
\end{equation}
Assume $\lambda+\mu\ne0$.  Set $r_\pm=\pm\tau$ and
$\zeta_\pm=\zeta(r_\pm)$.  Each root $r_\pm$ has algebraic multiplicity three and
\[
 \ker\mathscr P_\eta(r_\pm)=\zeta_\pm^\perp,
 \qquad
 \dim\ker\mathscr P_\eta(r_\pm)=2.
\]
With
\begin{equation*}
 \mathfrak c_{\lambda,\mu}:=\frac{\lambda+3\mu}{\lambda+\mu},
 \qquad
 \bfa_0:=\zeta_\pm,
 \qquad
 \bfa_1:=-\mathfrak c_{\lambda,\mu}e_3,
\end{equation*}
one has the length-two chain relations
\begin{equation}
\label{eq:Jordan-chain-relations}
 \mathscr P_\eta(r_\pm)\bfa_0=\bm{0},
 \qquad
 \mathscr P_\eta(r_\pm)\bfa_1
 +\mathscr P_\eta'(r_\pm)\bfa_0=\bm{0}.
\end{equation}
Consequently,
\begin{equation}
\label{eq:generalized-static-mode}
 \bfw_{\eta,\pm}(x)
 :=\left(x_3\zeta_\pm-\mathfrak c_{\lambda,\mu}e_3\right)
 e^{i\eta\cdot x'+r_\pm x_3}
\end{equation}
belongs to $\cH_{\lambda,\mu}$.  The two partial multiplicities at $r_\pm$ are
$2$ and $1$; in particular, no Jordan chain at this root has length greater than two.

If $\lambda+\mu=0$, then
$\mathscr P_\eta(r)=\mu(r^2-\tau^2)I_3$; each root has algebraic and geometric
multiplicity three and is semisimple.
\end{lemma}

\begin{proof}
Formula \eqref{eq:normal-symbol-action} follows by substituting the exponential ansatz into
the static Lam\'e equation.  Since
$\zeta(r)\cdot\zeta(r)=r^2-\tau^2$, the matrix determinant lemma gives, whenever
$r^2\ne\tau^2$,
\begin{align*}
 \det\mathscr P_\eta(r)
 &=[\mu(r^2-\tau^2)]^3
 \left(1+\frac{\lambda+\mu}{\mu}\right)\\
 &=\mu^2(\lambda+2\mu)(r^2-\tau^2)^3.
\end{align*}
Both sides are polynomials, so \eqref{eq:normal-determinant} holds for all $r$.
At $r=r_\pm$ and $\lambda+\mu\ne0$,
\[
 \mathscr P_\eta(r_\pm)
 =(\lambda+\mu)\zeta_\pm\zeta_\pm^\top,
\]
which has rank one and kernel $\zeta_\pm^\perp$.  The order of the zero in
\eqref{eq:normal-determinant} is three, proving the algebraic-multiplicity assertion.

The first relation in \eqref{eq:Jordan-chain-relations} follows from
$\zeta_\pm\cdot\zeta_\pm=0$.  Since
\[
 \mathscr P_\eta'(r)
 =2\mu rI_3+(\lambda+\mu)
 \left(e_3\zeta(r)^\top+\zeta(r)e_3^\top\right),
\]
one obtains
\[
 \mathscr P_\eta'(r_\pm)\bfa_0
 =(\lambda+3\mu)r_\pm\zeta_\pm
\]
and
\[
 \mathscr P_\eta(r_\pm)\bfa_1
 =-(\lambda+\mu)\mathfrak c_{\lambda,\mu}
   r_\pm\zeta_\pm
 =-(\lambda+3\mu)r_\pm\zeta_\pm.
\]
This proves the second chain relation.  For a quadratic matrix polynomial,
\[
 \mathscr P_\eta(\partial_3)
 \left[(x_3\bfa_0+\bfa_1)e^{r_\pm x_3}\right]
 =e^{r_\pm x_3}
 \left[x_3\mathscr P_\eta(r_\pm)\bfa_0
       +\mathscr P_\eta(r_\pm)\bfa_1
       +\mathscr P_\eta'(r_\pm)\bfa_0\right].
\]
Together with the horizontal exponential, this proves
\eqref{eq:generalized-static-mode}.  The leading coefficient of $\mathscr P_\eta$ is
$\operatorname{diag}(\mu,\mu,\lambda+2\mu)$ and is invertible, so the matrix polynomial is
regular.  Standard matrix-polynomial theory \cite{GohbergLancasterRodman2009} gives exactly
$\dim\ker\mathscr P_\eta(r_\pm)=2$ Jordan chains at $r_\pm$, whose positive lengths sum
to the algebraic multiplicity three.  The displayed relations produce a chain of length at
least two.  The two chain lengths are therefore exactly $2$ and $1$, and the length-two
chain cannot be extended.  The last assertion follows immediately when
$\lambda+\mu=0$.
\end{proof}

Thus the two normal roots and their first generalized chains supply, for a separated scalar
profile multiplied by a fixed vector, one transform-value equation and one first-derivative
equation per sign.  Four is therefore the largest number of profile equations furnished by
this particular normal-chain construction.  We do not claim maximality among all globally
defined static Lam\'e probes.

\begin{lemma}
\label{lem:multi-profile-completeness}
Let $\bfM\in\R^3\setminus\{\bm{0}\}$ and suppose
$\theta_1,\ldots,\theta_N\in L^2_c(D;\R)$ and
\[
 \bfH(x)=\bfM\sum_{k=1}^N\theta_k(x')\psi_k(x_3).
\]
Assume
\begin{equation}
\label{eq:multi-profile-orthogonality}
 \int_{\R^3}\bfH(x)\cdot\bfv(x)\,dx=0
 \qquad\forall\bfv\in\cH_{\lambda,\mu}.
\end{equation}
Then $\theta_1=\cdots=\theta_N=0$ a.e. in $\R^2$ under either of the following
conditions:
\begin{enumerate}
\item[(i)] $1\le N\le2$ and $\Psi=(\psi_1,\ldots,\psi_N)$ is two-sided Laplace
nondegenerate;
\item[(ii)] $1\le N\le4$, $\lambda+\mu\ne0$, and $\Psi$ is Hermite--Laplace
nondegenerate.
\end{enumerate}
\end{lemma}

\begin{proof}
Fix $\eta\in\R^2\setminus\{\bm{0}\}$, write $\tau=|\eta|$, and set
$\zeta_\pm=(i\eta_1,i\eta_2,\pm\tau)$.  The real vector $\bfM$ cannot belong to
$\operatorname{span}_{\C}\{\zeta_\pm\}$, since otherwise
$|\bfM|^2=\bfM\cdot\bfM=0$.  Hence there is
$\bfa_\pm\in\zeta_\pm^\perp$ with $\bfM\cdot\bfa_\pm\ne0$.  Testing
\eqref{eq:multi-profile-orthogonality} with
$\bfa_\pm e^{\zeta_\pm\cdot x}$ gives
\begin{equation}
\label{eq:profile-value-equations}
 \sum_{k=1}^N\widehat{\theta_k}(-\eta)
 \mathscr L_{\psi_k}(\pm\tau)=0.
\end{equation}
Under condition (i), these equations form
\[
 \mathscr V_\Psi(\tau)
 \begin{pmatrix}
  \widehat{\theta_1}(-\eta)\\[-1mm]
  \vdots\\[-1mm]
  \widehat{\theta_N}(-\eta)
 \end{pmatrix}=0.
\]
They force all Fourier coefficients to vanish whenever
$\tau\in\mathcal S_\Psi^{(0)}$.

Assume now condition (ii).  Testing with the generalized modes
\eqref{eq:generalized-static-mode} gives
\begin{align*}
0={}&(\bfM\cdot\zeta_\pm)
 \sum_{k=1}^N\widehat{\theta_k}(-\eta)
       \mathscr L'_{\psi_k}(\pm\tau)
 \\
&-\mathfrak c_{\lambda,\mu}M_3
 \sum_{k=1}^N\widehat{\theta_k}(-\eta)
       \mathscr L_{\psi_k}(\pm\tau).
\end{align*}
The second sum vanishes by \eqref{eq:profile-value-equations}.  Moreover,
\[
 \bfM\cdot\zeta_\pm=i\bfM'\cdot\eta\pm M_3\tau,
\qquad \bfM':=(M_1,M_2).
\]
If $M_3\ne0$, this factor is nonzero for every $\eta\ne\bm{0}$; if $M_3=0$, it vanishes only
on the line $\bfM'\cdot\eta=0$.  Hence, on an open dense set of horizontal frequencies,
\begin{equation*}
 \sum_{k=1}^N\widehat{\theta_k}(-\eta)
       \mathscr L'_{\psi_k}(\pm\tau)=0.
\end{equation*}
Together with \eqref{eq:profile-value-equations}, this is the system
\[
 \mathscr H_\Psi(\tau)
 \begin{pmatrix}
  \widehat{\theta_1}(-\eta)\\[-1mm]
  \vdots\\[-1mm]
  \widehat{\theta_N}(-\eta)
 \end{pmatrix}=0
\]
whenever $\tau\in\mathcal S_\Psi^{(1)}$ and the polarization factor is nonzero.

In case (i), the set
$\{\eta\ne\bm{0}:|\eta|\in\mathcal S_\Psi^{(0)}\}$ is dense in $\R^2$.  In case (ii),
intersecting the analogous radial set for $\mathcal S_\Psi^{(1)}$ with the complement of
the exceptional line $\bfM'\cdot\eta=0$ (when $M_3=0$) still leaves a dense subset of
$\R^2$.  The functions
$\theta_k$ are compactly supported in $L^2$, so their Fourier transforms are continuous.
They vanish on a dense set and hence everywhere.  Fourier injectivity proves the result.
\end{proof}

\begin{proof}[Proof of Proposition~\ref{prop:translated-Hermite}]
Write $\Lambda(z)=\mathscr L_\varphi(z)$.  Then
\[
 \mathscr L_{\psi_k}(z)=e^{d_kz}\Lambda(z),
 \qquad
 \mathscr L'_{\psi_k}(z)=e^{d_kz}\big(d_k\Lambda(z)+\Lambda'(z)\big).
\]
For $i<j$,
\[
 \det\mathscr V_{(\psi_i,\psi_j)}(z)
 =2\sinh((d_i-d_j)z)\Lambda(z)\Lambda(-z)\not\equiv0,
\]
so every one- or two-profile subfamily is two-sided Laplace nondegenerate.

For the four-column Hermite matrix, expand the second and fourth rows by linearity.
Terms containing $\Lambda'(z)$ or $\Lambda'(-z)$ have two proportional rows.  Their
determinants vanish.  Factoring the remaining rows gives
\begin{equation}
\label{eq:translated-Hermite-factorization}
 \det\mathscr H_{(\psi_1,\ldots,\psi_4)}(z)
 =\Lambda(z)^2\Lambda(-z)^2D_d(z),
\end{equation}
where
\[
 D_d(z)
 :=\det
 \begin{pmatrix}
 e^{d_1z}&e^{d_2z}&e^{d_3z}&e^{d_4z}\\
 d_1e^{d_1z}&d_2e^{d_2z}&d_3e^{d_3z}&d_4e^{d_4z}\\
 e^{-d_1z}&e^{-d_2z}&e^{-d_3z}&e^{-d_4z}\\
 d_1e^{-d_1z}&d_2e^{-d_2z}&d_3e^{-d_3z}&d_4e^{-d_4z}
 \end{pmatrix}.
\]
Expand by two columns in the upper block and the complementary columns in the lower block.
The term with upper columns $\{3,4\}$ and lower columns $\{1,2\}$ is
\[
 \pm(d_4-d_3)(d_2-d_1)
 e^{(d_3+d_4-d_1-d_2)z}.
\]
For every other two-element set $S\ne\{3,4\}$,
\[
 \sum_{j\in S}d_j-\sum_{j\notin S}d_j
 <d_3+d_4-d_1-d_2.
\]
Thus this term has the unique largest exponential rate as $z\to+\infty$.  Hence
\[
 D_d(z)
 =\pm(d_4-d_3)(d_2-d_1)
 e^{(d_3+d_4-d_1-d_2)z}(1+o(1)).
\]
Hence the determinant in \eqref{eq:translated-Hermite-factorization} is not identically
zero.  A nontrivial entire function cannot vanish on the whole positive real axis, so the
full four-column matrix is invertible on an open dense set of positive radii.  At each such
radius, every subfamily of columns is linearly independent.  This proves the claim.
\end{proof}

\subsection{Kernel of the reduced constant-vector transform}

Nonradiating elastic sources and their geometry have been studied, for example, in
\cite{BlastenLin2019}.  The following elementary construction records the limitation of the
specific static transform used later.

\begin{proposition}
\label{prop:general-density-obstruction}
For every $\bfM\in\R^3\setminus\{\bm{0}\}$, there are infinitely many linearly independent
$\omega\in C_c^\infty(Q)$ such that
\begin{equation}
\label{eq:reduced-kernel}
 \int_Q\omega(x)\bfM\cdot\bfv(x)\,dx=0
 \qquad\forall\bfv\in\cH_{\lambda,\mu}.
\end{equation}
\end{proposition}

\begin{proof}
If $\bfv\in\cH_{\lambda,\mu}$, then taking the divergence of the Lam\'e equation and then
applying $\Delta$ gives $\Delta^2\bfv=0$.  Thus $\bfM\cdot\bfv$ is biharmonic.  For any
$\omega_0\in C_c^\infty(Q)$ with $\Delta^2\omega_0\not\equiv0$, set
$\omega=\Delta^2\omega_0$.  Two
integrations by parts give
\[
 \int_Q\omega\,\bfM\cdot\bfv\,dx
 =\int_Q\omega_0\,\Delta^2(\bfM\cdot\bfv)\,dx=0.
\]
Choosing functions $\omega_0$ with mutually disjoint supports yields an infinite linearly
independent family.
\end{proof}

Proposition~\ref{prop:general-density-obstruction} concerns only the transform
\eqref{eq:reduced-kernel}.  It neither constructs two configurations with the same passive
data nor characterizes the nullspace of the full pair of weights in
\eqref{eq:intro-density-identities}.

\section{Laplace reduction and identification of the weighted initial state}
\label{sec:laplace}

\subsection{Variational and integral formulations}

For $s>0$, define
\[
 \widetilde{\bfu}(s):=\int_0^\infty e^{-st}\bfu(t)\,dt
 \quad\text{in }H^1(\R^3)^3.
\]
The polynomial bound \eqref{eq:H1-growth} justifies the Bochner integral.  Let
$V=H^1(\R^3)^3$ and use $H_\rho$ as the pivot space.  The abstract energy equation gives
$\pa_t^2\bfu\in C([0,\infty);V')$.  Integrating first on $(0,T)$ in $V'$ gives
\begin{align*}
 \int_0^T e^{-st}\pa_t^2\bfu(t)\,dt
 ={}&e^{-sT}\big(\pa_t\bfu(T)+s\bfu(T)\big)-\bfg-s\bff\\
 &+s^2\int_0^T e^{-st}\bfu(t)\,dt.
\end{align*}
By \eqref{eq:H1-growth}, the first term tends to zero in $V'$ as $T\to\infty$.
Consequently,
\[
 \int_0^\infty e^{-st}\pa_t^2\bfu(t)\,dt
 =s^2\widetilde{\bfu}(s)-s\bff-\bfg.
\]
Testing this identity against a fixed field in $V$ in the weak wave equation yields
\begin{equation}
\label{eq:background-laplace}
 (\cA_{\lambda,\mu}+\rho_0s^2)\widetilde{\bfu}
 =\bfG+s\bfF-s^2q\widetilde{\bfu}.
\end{equation}
Equivalently, $\widetilde{\bfu}$ is the unique element of $H^1(\R^3)^3$ satisfying
\begin{equation}
\label{eq:laplace-variational}
 a(\widetilde{\bfu},\bfv)
 +s^2\int_{\R^3}\rho\widetilde{\bfu}\cdot\overline{\bfv}\,dx
 =\int_{\R^3}(\bfG+s\bfF)\cdot\overline{\bfv}\,dx
 \quad\forall\bfv\in H^1(\R^3)^3.
\end{equation}
Indeed, the left-hand side is coercive by \eqref{eq:form-Fourier} and the positive mass
term, so uniqueness follows from Lax--Milgram.

Set
\begin{equation*}
 \kappa_{\mathrm S}=s\sqrt{\frac{\rho_0}{\mu}},
 \qquad
 \kappa_{\mathrm P}=s\sqrt{\frac{\rho_0}{\lambda+2\mu}},
\end{equation*}
and let $\Phi_\kappa(x)=e^{-\kappa|x|}/(4\pi|x|)$.  The modified Kupradze tensor is
\begin{equation}
\label{eq:modified-kupradze}
 \bfGamma_s(x)
 =\frac1\mu\Phi_{\kappa_{\mathrm S}}(x)I_3
 -\frac1{\rho_0s^2}\nabla\nabla^\top
  \left(\Phi_{\kappa_{\mathrm S}}(x)-\Phi_{\kappa_{\mathrm P}}(x)\right).
\end{equation}
Its Fourier multiplier is
\begin{equation}
\label{eq:Gamma-symbol}
 \widehat{\bfGamma_s}(\xi)
 =\frac{P_T(\xi)}{\mu|\xi|^2+\rho_0s^2}
  +\frac{P_L(\xi)}{(\lambda+2\mu)|\xi|^2+\rho_0s^2},
\end{equation}
so
\begin{equation*}
 (\cA_{\lambda,\mu}+\rho_0s^2)\bfGamma_s=\delta_0I_3.
\end{equation*}
Applying this resolvent to \eqref{eq:background-laplace} gives the
Lippmann--Schwinger equation
\begin{equation}
\label{eq:Lippmann-Schwinger}
 \widetilde{\bfu}
 =\bfGamma_s*(\bfG+s\bfF)-s^2\bfGamma_s*(q\widetilde{\bfu}).
\end{equation}
The equivalence between \eqref{eq:laplace-variational} and
\eqref{eq:Lippmann-Schwinger} is recorded in Appendix~\ref{app:variational}.

\subsection{Zero-frequency expansion of the modified Kupradze tensor}

Fix $R>0$ with $\overline\Om\subset B_R$.  The kernel $\bfK_2$ below is given explicitly
in \eqref{eq:K2}; it extends continuously to the origin by setting $\bfK_2(0)=0$.

\begin{lemma}
\label{lem:kernel-expansion}
There are $s_0>0$ and $C_R>0$ such that, for
$0<s<s_0$ and $0<|x|\le2R$,
\begin{equation*}
 \bfGamma_s(x)
 =\bfGamma_0(x)-s\gamma_1I_3+s^2\bfK_2(x)+\bfE_s(x),
 \qquad
 |\bfE_s(x)|\le C_Rs^3|x|^2.
\end{equation*}
Consequently, if $\bfh\in L^2(\R^3)^3$ and
$\esssupp\bfh\subset\overline Q$, then
\begin{equation}
\label{eq:operator-expansion}
 \bfGamma_s*\bfh
 =\bfGamma_0*\bfh-s\gamma_1\int_Q\bfh(y)\,dy
  +s^2\bfK_2*\bfh
  +O_{L^\infty(B_R)}(s^3\|\bfh\|_{L^1(Q)}).
\end{equation}
Moreover, uniformly for $0<s<s_0$,
\begin{align}
\label{eq:L2-Linf}
 \|\bfGamma_s*\bfh\|_{L^\infty(B_R)}
 &\le C_R\|\bfh\|_{L^2(Q)},\\
\label{eq:B-bound}
 \|\bfGamma_s*(q\bfv)\|_{L^\infty(B_R)}
 &\le C_R\|q\|_{L^\infty(Q)}\|\bfv\|_{L^\infty(B_R)},
\end{align}
whenever the displayed expressions are defined and the factors are supported in
$\overline Q$.  Finally,
\begin{equation}
\label{eq:first-order-operator-bound}
 \left\|\bfGamma_s*\bfh-\bfGamma_0*\bfh
       +s\gamma_1\int_Q\bfh(y)\,dy\right\|_{L^\infty(B_R)}
 \le C_Rs^2\|\bfh\|_{L^2(Q)}.
\end{equation}
\end{lemma}

\begin{proof}
The expansion and the differentiated remainder are proved in Appendix~\ref{app:kernel}.
The only singular term is the Kelvin tensor, of order $|x|^{-1}$.  Thus
$|\bfGamma_s(x)|\le C_R|x|^{-1}$ on $0<|x|\le2R$.  Since $|x|^{-2}$ is locally
integrable in three dimensions, Cauchy--Schwarz gives \eqref{eq:L2-Linf}; integration
against $|x-y|^{-1}$ gives \eqref{eq:B-bound}.  Subtracting the first two terms of
\eqref{eq:operator-expansion}, using the local boundedness of $\bfK_2$, and applying
$L^1(Q)\hookleftarrow L^2(Q)$ gives \eqref{eq:first-order-operator-bound}.
\end{proof}

\subsection{Uniform expansion of the transformed field}

\begin{proposition}
\label{prop:field-expansion}
Let $q\in L^\infty(\R^3)$ satisfy $\esssupp q\subset\overline Q$ and assume
\begin{equation}
\label{eq:field-positive-density}
 \rho_-\le\rho_0+q\le\rho_+.
\end{equation}
Let $\bfF,\bfG\in L^2_c(Q)^3$ and fix $R$ with $\overline\Om\subset B_R$.  For all
sufficiently small $s>0$, the restriction of \eqref{eq:Lippmann-Schwinger} to $B_R$ has a
unique fixed point in $L^\infty(B_R)^3$, equal to the variational solution.  It satisfies
\begin{equation}
\label{eq:field-expansion}
 \widetilde{\bfu}
 =\bfU_0+s\bfU_1+s^2\bfU_2+O_{L^\infty(B_R)}(s^3),
\end{equation}
where
\begin{equation*}
 \bfU_0=\bfGamma_0*\bfG,
 \qquad
 \bfU_1=\bfGamma_0*\bfF-\gamma_1\int_Q\bfG(y)\,dy,
\end{equation*}
and
\begin{equation}
\label{eq:U2}
 \bfU_2
 =\bfK_2*\bfG-\gamma_1\int_Q\bfF(y)\,dy-\bfGamma_0*(q\bfU_0).
\end{equation}
More precisely, for every $K_q,K_F,K_G>0$, there are
\[
 s_*=s_*(R,Q,\lambda,\mu,\rho_0,K_q)>0,
 \qquad
 C_*=C_*(R,Q,\lambda,\mu,\rho_0,K_q,K_F,K_G)>0,
\]
such that the remainder in \eqref{eq:field-expansion} is bounded by $C_*s^3$ whenever
\[
 \|q\|_{L^\infty(Q)}\le K_q,\qquad
 \|\bfF\|_{L^2(Q)}\le K_F,\qquad
 \|\bfG\|_{L^2(Q)}\le K_G,
 \qquad 0<s<s_*.
\]
\end{proposition}

\begin{proof}
Define on $L^\infty(B_R)^3$
\[
 \cB_s\bfv:=\bfGamma_s*(q\bfv).
\]
Lemma~\ref{lem:kernel-expansion} gives a uniform operator bound for $\cB_s$.  Hence
$\mathrm{Id}+s^2\cB_s$ is invertible for small $s$.  The exact identity
\begin{equation}
\label{eq:inverse-expansion}
 (\mathrm{Id}+s^2\cB_s)^{-1}-(\mathrm{Id}-s^2\cB_s)
 =s^4\cB_s^2(\mathrm{Id}+s^2\cB_s)^{-1}
\end{equation}
shows directly that the remainder is $O(s^4)$ uniformly in operator norm on every bounded
coefficient class.  The kernel expansion gives
\begin{align}
 \bfGamma_s*(\bfG+s\bfF)
 ={}&\bfGamma_0*\bfG
 +s\left(\bfGamma_0*\bfF-\gamma_1\int_Q\bfG(y)\,dy\right)
 \nonumber\\
 &+s^2\left(\bfK_2*\bfG-\gamma_1\int_Q\bfF(y)\,dy\right)
 +O_{L^\infty(B_R)}(s^3).
\label{eq:free-expansion}
\end{align}
Furthermore, \eqref{eq:first-order-operator-bound} implies
$\|\cB_s-\cB_0\|=O(s)$, where
$\cB_0\bfv=\bfGamma_0*(q\bfv)$.  Multiplying
\eqref{eq:inverse-expansion} and \eqref{eq:free-expansion}, one has
\begin{equation*}
 -s^2\cB_s\bfU_0
 =-s^2\cB_0\bfU_0-s^2(\cB_s-\cB_0)\bfU_0
 =-s^2\cB_0\bfU_0+O_{L^\infty(B_R)}(s^3),
\end{equation*}
while $-s^3\cB_s\bfU_1=O_{L^\infty(B_R)}(s^3)$ and the inverse remainder in
\eqref{eq:inverse-expansion} is $O(s^4)$ in operator norm.  Thus the only order-$s^2$
contribution from the inverse is $-s^2\cB_0\bfU_0$, which proves
\eqref{eq:field-expansion}--\eqref{eq:U2}.

We record the uniformity because it is needed in the difference expansion below.  If
$\|q\|_{L^\infty(Q)}\le K_q$, then \eqref{eq:B-bound} gives
$\|\cB_s\|\le C_RK_q$ for $0<s<s_0$.  Hence $s_*$ may be chosen so that
$s_*^2C_RK_q\le1/2$.  The convolution estimates yield
\[
 \|\bfU_0\|_{L^\infty(B_R)}\le C_RK_G,
 \qquad
 \|\bfU_1\|_{L^\infty(B_R)}\le C_R(K_F+K_G),
\]
and, using \eqref{eq:U2},
\[
 \|\bfU_2\|_{L^\infty(B_R)}
 \le C_R(K_G+K_F+K_qK_G).
\]
The remainder in the free expansion is bounded by $C_Rs^3(K_F+K_G)$, while the omitted
terms in the Neumann product contain either this factor or at least $s^4\|\cB_s\|^2$.
Thus all order-$s^3$ contributions are bounded by
$C_*s^3$ with $C_*$ depending only on the quantities stated in the proposition.  The
positivity hypothesis \eqref{eq:field-positive-density} identifies the fixed point with the
physical variational solution through Appendix~\ref{app:variational}.
\end{proof}

\subsection{Exterior equality and proof of Theorem~\ref{thm:weighted}}

\begin{lemma}
\label{lem:exterior-equality}
If \eqref{eq:data-equality} holds, then for every $s>0$,
\begin{equation*}
 \widetilde{\bfu}_1(\cdot,s)=\widetilde{\bfu}_2(\cdot,s)
 \quad\text{in }\R^3\setminus\overline\Om.
\end{equation*}
\end{lemma}

\begin{proof}
The continuous trace map commutes with the Bochner--Laplace integral, so the transformed
fields have the same Dirichlet trace on $\pa\Om$.  Their difference $\bfW$ satisfies
\[
 (\cA_{\lambda,\mu}+\rho_0s^2)\bfW=\bm{0}
 \quad\text{in }\Om^e:=\R^3\setminus\overline\Om,
 \qquad
 \bfW|_{\pa\Om}=\bm{0}.
\]
For an exterior Lipschitz domain, the zero-trace characterization of $H^1_0$ gives
$\bfW\in H^1_0(\Om^e)^3$.  Its zero extension $\bfW^e$ therefore belongs to
$H^1(\R^3)^3$.  The weak exterior equation, initially tested with compactly supported
fields, extends by density to $H^1_0(\Om^e)^3$.  Taking $\bfW$ itself as test field gives
\[
 a(\bfW^e,\bfW^e)+\rho_0s^2\|\bfW^e\|_{L^2(\R^3)}^2=0.
\]
Both terms are nonnegative by \eqref{eq:form-Fourier}; therefore $\bfW=\bm{0}$.
\end{proof}

\begin{proof}[Proof of Theorem~\ref{thm:weighted}]
Fix an arbitrary $R>0$ with $\overline\Om\subset B_R$.  Write
$\bfF_j=\rho_j\bff_j$ and $\bfG_j=\rho_j\bfg_j$.
Proposition~\ref{prop:field-expansion} gives, uniformly on bounded sets,
\begin{equation}
\label{eq:first-two-general}
 \widetilde{\bfu}_j
 =\bfGamma_0*\bfG_j
 +s\left(\bfGamma_0*\bfF_j-\gamma_1\int_Q\bfG_j(y)\,dy\right)
 +O(s^2).
\end{equation}
By Lemma~\ref{lem:exterior-equality}, the left-hand sides agree in the exterior.  Letting
$s\downarrow0$ in \eqref{eq:first-two-general} gives
\[
 \bfGamma_0*(\bfG_1-\bfG_2)=\bm{0}
 \quad\text{in }(\R^3\setminus\overline\Om)\cap B_R.
\]
Since $R$ is arbitrary, this identity holds in all of
$\R^3\setminus\overline\Om$.
Lemma~\ref{lem:potential-orthogonality} and the structure
\[
 \bfG_1-\bfG_2
 =(\bfG_1^{\flat}-\bfG_2^{\flat})(x')\psi_g(x_3)
\]
reduce the difference to Lemma~\ref{lem:completeness}; hence
$\bfG_1=\bfG_2$.  In particular, their vector means agree.

Subtract the common zeroth-order term in \eqref{eq:first-two-general}, divide by $s$, and
let $s\downarrow0$.  The mean terms cancel, and we obtain
\[
 \bfGamma_0*(\bfF_1-\bfF_2)=\bm{0}
 \quad\text{in }(\R^3\setminus\overline\Om)\cap B_R.
\]
Again, arbitrariness of $R$ extends the identity to the whole exterior.
The same potential-orthogonality and completeness argument, now with $\psi_f$, gives
$\bfF_1=\bfF_2$.
\end{proof}

\section{Density identities and uniqueness consequences}
\label{sec:density}

Assume throughout this section that the hypotheses of Theorem~\ref{thm:weighted} and
the data equality \eqref{eq:data-equality} hold.  Let $\bfF,\bfG$ be the resulting common
weighted initial states, let $q_j=\rho_j-\rho_0$, and set
\[
 \delta\rho:=\rho_1-\rho_2=q_1-q_2.
\]
The fields $\bfU_0$ and $\bfU_1$ are those in \eqref{eq:common-U01}.

\subsection{The two static density identities}

\begin{proposition}
\label{prop:s2-density}
One has
\begin{equation}
\label{eq:s2-potential-zero}
 \bfGamma_0*(\delta\rho\,\bfU_0)=\bm{0}
 \quad\text{in }\R^3\setminus\overline\Om,
\end{equation}
and therefore
\begin{equation}
\label{eq:I0}
 \int_Q\delta\rho(x)\bfU_0(x)\cdot\bfv(x)\,dx=0
 \qquad\forall\bfv\in\cH_{\lambda,\mu}.
\end{equation}
In particular,
\begin{equation}
\label{eq:moment-U0}
 \int_Q\delta\rho(x)\bfU_0(x)\,dx=\bm{0}.
\end{equation}
\end{proposition}

\begin{proof}
Fix an arbitrary $R>0$ with $\overline\Om\subset B_R$.  For $j=1,2$,
Proposition~\ref{prop:field-expansion} gives
\begin{align}
 \widetilde{\bfu}_j
 ={}&\bfU_0+s\bfU_1
 \nonumber\\
 &+s^2\left[
 \bfK_2*\bfG-\gamma_1\int_Q\bfF(y)\,dy-\bfGamma_0*(q_j\bfU_0)
 \right]
 +O_{L^\infty(B_R)}(s^3).
\label{eq:uj-through-s2}
\end{align}
The two transformed fields agree on the exterior by
Lemma~\ref{lem:exterior-equality}.  Subtracting \eqref{eq:uj-through-s2}, dividing by
$s^2$, and letting $s\downarrow0$ gives \eqref{eq:s2-potential-zero} on
$(\R^3\setminus\overline\Om)\cap B_R$.  Since $R$ is arbitrary, it holds in the whole
exterior.  Moreover, $\delta\rho\,\bfU_0\in L^2(\R^3)^3$ and
$\esssupp(\delta\rho\,\bfU_0)\subset\overline Q$.  Hence
Lemma~\ref{lem:potential-orthogonality} gives
\eqref{eq:I0}.  Taking the three constant static fields yields \eqref{eq:moment-U0}.
\end{proof}

The next identity requires one additional order.  Its proof starts from the exact difference
of the transformed integral equations.

\begin{lemma}
\label{lem:difference-expansion}
Fix $R>0$ with $\overline\Om\subset B_R$ and numbers $K_q,K_F,K_G>0$.  There exist
$s_R=s_R(R,Q,\lambda,\mu,\rho_0,K_q)>0$ and
$C_R=C_R(R,Q,\lambda,\mu,\rho_0,K_q,K_F,K_G)>0$ such that, whenever
\[
 \|q_j\|_{L^\infty(Q)}\le K_q,\qquad
 \|\bfF\|_{L^2(Q)}\le K_F,\qquad
 \|\bfG\|_{L^2(Q)}\le K_G,
\]
one has
\begin{align}
 \widetilde{\bfu}_1-\widetilde{\bfu}_2
 ={}&-s^2\bfGamma_0*(\delta\rho\,\bfU_0)
 \nonumber\\
 &-s^3\left[
 \bfGamma_0*(\delta\rho\,\bfU_1)
 -\gamma_1\int_Q\delta\rho(y)\bfU_0(y)\,dy
 \right]
 +\bfR_R(s),
\label{eq:difference-s3}
\end{align}
where
\[
 \|\bfR_R(s)\|_{L^\infty(B_R)}\le C_Rs^4,
 \qquad 0<s<s_R.
\]
\end{lemma}

\begin{proof}
The free term in \eqref{eq:Lippmann-Schwinger} is common to both configurations.  Hence
\begin{equation}
\label{eq:exact-difference}
 \widetilde{\bfu}_1-\widetilde{\bfu}_2
 =-s^2\bfGamma_s*
  \left(q_1\widetilde{\bfu}_1-q_2\widetilde{\bfu}_2\right).
\end{equation}
By Proposition~\ref{prop:field-expansion}, uniformly under the bounds in the statement,
\[
 \widetilde{\bfu}_j
 =\bfU_0+s\bfU_1+s^2\bfU_{2,j}
  +O_{L^\infty(B_R)}(s^3),
\]
where
\[
 \bfU_{2,j}
 =\bfK_2*\bfG-\gamma_1\int_Q\bfF(y)\,dy-\bfGamma_0*(q_j\bfU_0).
\]
Consequently,
\begin{equation}
\label{eq:difference-source-expansion}
 q_1\widetilde{\bfu}_1-q_2\widetilde{\bfu}_2
 =\bfh_0+s\bfh_1+s^2\bfZ_s,
 \qquad
 \bfh_0:=\delta\rho\,\bfU_0,
 \quad
 \bfh_1:=\delta\rho\,\bfU_1,
\end{equation}
with
\[
 \bfZ_s=q_1\bfU_{2,1}-q_2\bfU_{2,2}+O_{L^2(Q)}(s).
\]
The bounds in Proposition~\ref{prop:field-expansion} imply
\begin{equation}
\label{eq:rs-uniform-bound}
 \sup_{0<s<s_R}\|\bfZ_s\|_{L^2(Q)}
 \le C(R,Q,\lambda,\mu,\rho_0,K_q,K_F,K_G).
\end{equation}
In the same way, $\|\bfh_0\|_{L^2(Q)}+\|\bfh_1\|_{L^2(Q)}$ is bounded by a constant with
the same parameter dependence.

Apply \eqref{eq:first-order-operator-bound} to the source in
\eqref{eq:difference-source-expansion}.  We obtain
\begin{align*}
 \bfGamma_s*(\bfh_0+s\bfh_1+s^2\bfZ_s)
 ={}&\bfGamma_0*(\bfh_0+s\bfh_1+s^2\bfZ_s)\\
 &-s\gamma_1\int_Q(\bfh_0(y)+s\bfh_1(y)+s^2\bfZ_s(y))\,dy
 +\bfE_s^{\sharp},
\end{align*}
where, by \eqref{eq:first-order-operator-bound} and the uniform bound
\eqref{eq:rs-uniform-bound}, the remainder satisfies
\[
 \|\bfE_s^{\sharp}\|_{L^\infty(B_R)}
 \le C_Rs^2\|\bfh_0+s\bfh_1+s^2\bfZ_s\|_{L^2(Q)}
 \le C_Rs^2.
\]
Moreover, \eqref{eq:L2-Linf} and the boundedness of $Q$ give
\[
 s^2\|\bfGamma_0*\bfZ_s\|_{L^\infty(B_R)}
 +s^2\left|\int_Q\bfh_1(y)\,dy\right|
 +s^3\left|\int_Q\bfZ_s(y)\,dy\right|
 \le C_Rs^2.
\]
Therefore
\[
 \bfGamma_s*
 \left(q_1\widetilde{\bfu}_1-q_2\widetilde{\bfu}_2\right)
 =\bfGamma_0*\bfh_0
 +s\left(\bfGamma_0*\bfh_1-\gamma_1\int_Q\bfh_0(y)\,dy\right)
 +O_{L^\infty(B_R)}(s^2),
\]
with a constant depending only on the parameters displayed in the lemma.  Multiplying by
the prefactor $-s^2$ in \eqref{eq:exact-difference} proves
\eqref{eq:difference-s3} and the uniform $O(s^4)$ bound.
\end{proof}

\begin{proposition}
\label{prop:s3-density}
One has
\begin{equation}
\label{eq:s3-potential-zero}
 \bfGamma_0*(\delta\rho\,\bfU_1)=\bm{0}
 \quad\text{in }\R^3\setminus\overline\Om,
\end{equation}
and therefore
\begin{equation}
\label{eq:I1}
 \int_Q\delta\rho(x)\bfU_1(x)\cdot\bfv(x)\,dx=0
 \qquad\forall\bfv\in\cH_{\lambda,\mu}.
\end{equation}
\end{proposition}

\begin{proof}
Fix an arbitrary $R>0$ with $\overline\Om\subset B_R$.  By
Lemma~\ref{lem:difference-expansion}, the left-hand side of
\eqref{eq:difference-s3} is zero on the exterior.  The order-$s^2$
coefficient vanishes by Proposition~\ref{prop:s2-density}, and the constant vector in the
order-$s^3$ coefficient vanishes by \eqref{eq:moment-U0}.  Dividing by $s^3$ and letting
$s\downarrow0$ gives \eqref{eq:s3-potential-zero} on
$(\R^3\setminus\overline\Om)\cap B_R$.  Arbitrariness of $R$ gives the identity in the
whole exterior.  Since $\delta\rho\,\bfU_1\in L^2(\R^3)^3$ and
$\esssupp(\delta\rho\,\bfU_1)\subset\overline Q$,
Lemma~\ref{lem:potential-orthogonality} then yields \eqref{eq:I1}.
\end{proof}

\subsection{Finite-dimensional density families}

\begin{proof}[Proof of Theorem~\ref{thm:finite-rank}]
Write
\[
 \rho_j=\rho_0+\sum_{m=1}^M\alpha_{j,m}\phi_m,
 \qquad
 \delta\alpha_m=\alpha_{1,m}-\alpha_{2,m}.
\]
The identities \eqref{eq:I0} and \eqref{eq:I1}, tested with
$\bfv_1,\ldots,\bfv_L$, give
\[
 \mathbb M^{(0)}\delta\alpha=\bm{0},
 \qquad
 \mathbb M^{(1)}\delta\alpha=\bm{0}.
\]
Because $\delta\alpha\in\R^M$, these equations are equivalent to
$\mathbb M_{\R}\delta\alpha=\bm{0}$.  The real full-rank condition
\eqref{eq:real-full-rank} gives $\delta\alpha=\bm{0}$, hence $\rho_1=\rho_2$.
Theorem~\ref{thm:weighted} and positivity of the density give equality of the initial states.
\end{proof}

\subsection{The aligned multi-profile theorem}

\begin{lemma}
\label{lem:constant-vector-weight}
If $\bfF=\beta\bfG$ and $\bfM_G:=\int_Q\bfG(y)\,dy$, then
\begin{equation}
\label{eq:constant-weight-orthogonality}
 \int_Q\delta\rho(x)\bfM_G\cdot\bfv(x)\,dx=0
 \qquad\forall\bfv\in\cH_{\lambda,\mu}.
\end{equation}
\end{lemma}

\begin{proof}
By \eqref{eq:common-U01},
$\bfU_1=\beta\bfU_0-\gamma_1\bfM_G$.  Subtracting $\beta$ times \eqref{eq:I0} from
\eqref{eq:I1} and using $\gamma_1>0$ gives
\eqref{eq:constant-weight-orthogonality}.
\end{proof}

\begin{proof}[Proof of Theorem~\ref{thm:global}]
Take two triples in $\mathfrak A^{(N)}$ with the same passive data.
Theorem~\ref{thm:weighted}, applied with $\psi_f=\psi_g=\psi_{\mathrm{src}}$, gives common weighted
states $\bfF$ and $\bfG$.  Since each configuration satisfies
$\bfF_j=\beta_j\bfG_j$ and $\int_Q\bfG(y)\,dy\ne\bm{0}$, one has
$\beta_1=\beta_2=:\beta$.  Lemma~\ref{lem:constant-vector-weight} therefore applies with
the nonzero vector $\bfM_G=\int_Q\bfG(y)\,dy$.

Write
\[
 \rho_j(x)=\rho_0+\sum_{k=1}^Np_k^{(j)}(x')\psi_{\rho,k}(x_3)
\]
and set $\delta p_k=p_k^{(1)}-p_k^{(2)}$.  Then
\[
 \delta\rho(x)\bfM_G
 =\bfM_G\sum_{k=1}^N\delta p_k(x')\psi_{\rho,k}(x_3).
\]
The orthogonality \eqref{eq:constant-weight-orthogonality} and the corresponding alternative
of Lemma~\ref{lem:multi-profile-completeness} imply
$\delta p_1=\cdots=\delta p_N=0$.  Hence $\rho_1=\rho_2$.  Equality of the initial states
follows from equality of the weighted states and strict positivity of the common density.
\end{proof}

\begin{proof}[Proof of Proposition~\ref{prop:alignment-rigidity}]
Set $\bfH:=\bfF-c\bfG\in L^2_c(Q;\R^3)$.  By
\eqref{eq:common-U01},
\[
 \bfGamma_0*\bfH
 =\bfM\chi(x_3)+\bfC
 \quad\text{in }Q,
 \qquad
 \bfC:=\gamma_1\int_Q\bfG(y)\,dy.
\]
Applying $\cA_{\lambda,\mu}$ in distributions on $Q$ gives
\begin{equation}
\label{eq:rigidity-H}
 \bfH(x)
 =-\big(\mu\bfM+(\lambda+\mu)M_3e_3\big)\chi''(x_3).
\end{equation}
The constant vector multiplying $\chi''$ is nonzero: its horizontal part is
$\mu\bfM'$, while if $\bfM'=\bm{0}$ its third component is
$(\lambda+2\mu)M_3\ne0$.

Let $K'$ be the projection of $\esssupp\bfH$ onto $\R^2$.  Since
$\esssupp\bfH\subset D\times I$, one has $K'\subset D$ and hence
$\operatorname{dist}(K',\partial D)>0$.  Because $D$ is open and $K'$ is a proper compact
subset, there is a ball $B_0\subset D\setminus K'$.  Thus $\bfH=0$ in
$\mathcal D'(B_0\times I)^3$.  Choose $\vartheta_0\in C_c^\infty(B_0)$ with
$\int_{B_0}\vartheta_0\,dx'\ne0$.  Testing \eqref{eq:rigidity-H} against
$\vartheta_0(x')\vartheta_1(x_3)$, with arbitrary $\vartheta_1\in C_c^\infty(I)$, gives
\[
 -\big(\mu\bfM+(\lambda+\mu)M_3e_3\big)
 \left(\int_{B_0}\vartheta_0\,dx'\right)
 \langle\chi'',\vartheta_1\rangle=\bm{0}.
\]
The two prefactors are nonzero, so
$\langle\chi'',\vartheta_1\rangle=0$ for every $\vartheta_1$ and therefore
$\chi''=0$ in $\mathcal D'(I)$.  Equation \eqref{eq:rigidity-H} now gives
$\bfH=\bm{0}$, proving \eqref{eq:rigidity-alignment}.  Substitution into
\eqref{eq:common-U01} yields \eqref{eq:rigidity-constant}.  If
$\bfM\chi\not\equiv\bm{0}$, this constant is nonzero, and hence
$\int_Q\bfG(y)\,dy\ne\bm{0}$.
\end{proof}

If $\bfM_G=\bm{0}$, then $\bfU_1=\beta\bfU_0$ and the order-$s^3$ identity adds no
independent density direction.  The present theorem does not address that case.

\subsection{Conditional monotonicity}

\begin{proof}[Proof of Corollary~\ref{cor:monotonicity}]
Taking $\bfv=\bfv_*$ in \eqref{eq:I0} yields
\[
 \int_Q(\rho_1-\rho_2)(x)
       [\bfU_0(x)\cdot\bfv_*(x)]\,dx=0.
\]
The second factor is strictly positive a.e., while the first has one sign.  Thus
$\rho_1=\rho_2$ a.e. in $Q$, and the background condition gives equality on $\R^3$.
The initial states then agree by Theorem~\ref{thm:weighted}.
\end{proof}

The positive-probe hypothesis can be checked for a polarized nonnegative weighted velocity.

\begin{corollary}
Assume the hypotheses of Corollary~\ref{cor:monotonicity} except
\eqref{eq:positive-probe}.  If
\begin{equation*}
 \bfG(x)=\bfa G(x),
 \qquad
 \bfa\in\R^3\setminus\{\bm{0}\},
 \qquad
 G\ge0,
 \quad G\not\equiv0,
\end{equation*}
then \eqref{eq:positive-probe} holds with the constant field $\bfv_*\equiv\bfa$.
\end{corollary}

\begin{proof}
For $z\ne0$, the Kelvin formula \eqref{eq:kelvin} gives, with $e=z/|z|$,
\[
 \bfa\cdot\bfGamma_0(z)\bfa
 =\frac{(\lambda+3\mu)|\bfa|^2
       +(\lambda+\mu)(\bfa\cdot e)^2}
       {8\pi\mu(\lambda+2\mu)|z|}.
\]
If $\lambda+\mu\ge0$, the numerator is positive because
$\lambda+3\mu=(\lambda+2\mu)+\mu>0$.  If $\lambda+\mu<0$, it is bounded below by
$2(\lambda+2\mu)|\bfa|^2>0$.  Hence
\[
 \bfa\cdot\bfU_0(x)
 =\int_Q[\bfa\cdot\bfGamma_0(x-y)\bfa]G(y)\,dy>0
\]
for every $x$, since $G$ is nonnegative and nonzero on a set of positive measure.
\end{proof}

\section{Scope and open problems}
\label{sec:scope}

The argument first identifies the density-weighted initial state and then derives two static
Lam\'e transforms of the density difference.  Exact alignment converts a fixed linear
combination of those transforms into a constant-vector weight.  At each nonzero horizontal
frequency, the two normal roots give the profile-transform values at both signs.  When
$\lambda+\mu\ne0$, Lemma~\ref{lem:normal-pencil} shows that each root has partial
multiplicities $2$ and $1$; the exhibited length-two chain gives one additional
first-derivative equation.  Hence the two roots and their first generalized chains furnish
four scalar profile equations.  This explains the bound $N\le4$ for the present normal-chain
construction.  We do not claim that four is maximal among all possible globally defined
static Lam\'e probes.

The alignment $\bff=\beta\bfg$ remains a restrictive structural hypothesis.
Proposition~\ref{prop:alignment-rigidity} shows that the separated factorization
$\bfU_1-c\bfU_0=\bfM\chi(x_3)$ does not enlarge the class under the present
compact-support assumptions: it forces the same alignment.
Proposition~\ref{prop:general-density-obstruction} shows that the resulting constant-vector transform
has an infinite-dimensional kernel on unrestricted densities.  Neither statement rules out
larger identifiable classes based on the full spatially varying pair $(\bfU_0,\bfU_1)$.
A main open problem is injectivity of this two-weight transform on natural nonaligned
classes.

When $\lambda+\mu=0$, the normal matrix polynomial is
$\mu(r^2-|\eta|^2)I_3$ and the two roots are semisimple by
Lemma~\ref{lem:normal-pencil}.  The derivative-row construction is therefore unavailable,
but the one- and two-profile conclusions based on the pure value matrix remain valid for
two-sided Laplace nondegenerate profiles under the original ellipticity assumptions.  This
does not exclude other polynomial--exponential
static constructions outside the normal chains used here.

The limit $s\downarrow0$ uses the complete boundary history.  The present Laplace-domain
argument does not yield finite-time identification.  Stability is also unresolved at the
level of the passive map.  The gap \eqref{eq:Hermite-uniform-gap} controls only the
normalized profile matrix;
uniform control of the full static probing system additionally requires nondegenerate
polarization normalizations.  The smallest singular value in
Theorem~\ref{thm:finite-rank} likewise controls only the terminal finite-dimensional linear
system.  A remaining difficulty is stable extraction of the low-frequency coefficients
from noisy time-domain traces.

Further questions include the case $\int_Q\bfG(y)\,dy=\bm{0}$, higher Laplace orders
producing new source-dependent density weights, and simultaneous recovery when the Lam\'e
parameters are also unknown.  In the last problem, stiffness contrasts enter at the leading
order and the triangular reduction used here no longer applies.

\appendix

\section{Expansion of the modified Kupradze tensor}
\label{app:kernel}

We provide the kernel calculation underlying Lemma~\ref{lem:kernel-expansion}.  Let
$\varrho=|x|$.  To justify the differentiated remainder, write Taylor's formula in the form
\begin{equation*}
 e^{-z}
 =
 \sum_{m=0}^4\frac{(-z)^m}{m!}
 +
 z^5\mathfrak r_5(z),
 \qquad
 \mathfrak r_5(z)
 =
 -\frac1{4!}\int_0^1(1-\theta)^4e^{-\theta z}\,d\theta.
\end{equation*}
For every $z_0>0$,
\begin{equation}
\label{eq:r5-derivatives}
 \sup_{0\le z\le z_0}
 \left(
 |\mathfrak r_5(z)|
 +|\mathfrak r_5'(z)|
 +|\mathfrak r_5''(z)|
 \right)
 <\infty.
\end{equation}
Thus, for bounded $\varrho$ and small $\kappa$,
\begin{equation}
\label{eq:yukawa-series}
 \Phi_\kappa(x)
 =
 \frac1{4\pi\varrho}
 -\frac{\kappa}{4\pi}
 +\frac{\kappa^2\varrho}{8\pi}
 -\frac{\kappa^3\varrho^2}{24\pi}
 +\frac{\kappa^4\varrho^3}{96\pi}
 +
 \frac{\kappa^5\varrho^4}{4\pi}\mathfrak r_5(\kappa\varrho).
\end{equation}
We use the identities
\begin{equation*}
 \nabla\nabla^\top \varrho
 =
 \frac{I_3}{\varrho}-\frac{xx^\top}{\varrho^3},
 \qquad
 \nabla\nabla^\top \varrho^2=2I_3,
 \qquad
 \nabla\nabla^\top \varrho^3
 =
 3\varrho I_3+3\frac{xx^\top}{\varrho}.
\end{equation*}

Substitute \eqref{eq:yukawa-series} into \eqref{eq:modified-kupradze}.  The constant term
$1/(4\pi\varrho)$ cancels in the difference
$\Phi_{\kappa_{\mathrm S}}-\Phi_{\kappa_{\mathrm P}}$ before differentiation, while the terms linear in
$\kappa$ are spatial constants and disappear after the Hessian is applied.  Since
\begin{align*}
 \kappa_{\mathrm S}^2-\kappa_{\mathrm P}^2
 &={}
 \rho_0s^2
 \left[
 \mu^{-1}-(\lambda+2\mu)^{-1}
 \right]
 \\
 &={}
 \rho_0s^2
 \frac{\lambda+\mu}{\mu(\lambda+2\mu)},
\end{align*}
the order-$s^0$ coefficient is
\begin{align*}
 \frac{I_3}{4\pi\mu\varrho}
 -
 \frac{\lambda+\mu}{8\pi\mu(\lambda+2\mu)}
 \left(
 \frac{I_3}{\varrho}-\frac{xx^\top}{\varrho^3}
 \right).
\end{align*}
Collecting the coefficients of $I_3/\varrho$ gives
\[
 \frac{1}{4\pi\mu}
 -
 \frac{\lambda+\mu}{8\pi\mu(\lambda+2\mu)}
 =
 \frac{\lambda+3\mu}{8\pi\mu(\lambda+2\mu)},
\]
which proves the Kelvin formula \eqref{eq:kelvin}.

At order $s$, the first Yukawa term in \eqref{eq:modified-kupradze} contributes
\begin{equation}
\label{eq:linear-first}
 -s\frac{\sqrt{\rho_0}}{4\pi\mu^{3/2}}I_3.
\end{equation}
The cubic term in the Hessian difference contributes
\begin{equation}
\label{eq:linear-second}
 s\frac{\sqrt{\rho_0}}{12\pi}
 \left[
 \mu^{-3/2}-(\lambda+2\mu)^{-3/2}
 \right]I_3.
\end{equation}
Adding \eqref{eq:linear-first} and \eqref{eq:linear-second} yields
\[
 -s\frac{\sqrt{\rho_0}}{12\pi}
 \left[
 2\mu^{-3/2}+(\lambda+2\mu)^{-3/2}
 \right]I_3
 =
 -s\gamma_1 I_3,
\]
which proves \eqref{eq:gamma1}.

At order $s^2$, the first Yukawa term contributes
\[
 s^2\frac{\rho_0\varrho}{8\pi\mu^2}I_3.
\]
Using
$\nabla\nabla^\top \varrho^3=3\varrho I_3+3xx^\top/\varrho$, the quartic term in the Hessian part contributes
\[
 -s^2\frac{\rho_0}{32\pi}
 \left[
 \mu^{-2}-(\lambda+2\mu)^{-2}
 \right]
 \left(
 \varrho I_3+\frac{xx^\top}{\varrho}
 \right).
\]
Thus
\begin{equation}
\label{eq:K2}
 \bfK_2(x)
 =\frac{\rho_0|x|}{8\pi\mu^2}I_3
 -\frac{\rho_0}{32\pi}
 \left[\mu^{-2}-(\lambda+2\mu)^{-2}\right]
 \left(|x|I_3+\frac{xx^\top}{|x|}\right),
 \qquad x\ne0.
\end{equation}
Both terms tend to zero as $x\to0$, so $\bfK_2$ extends continuously by
$\bfK_2(0)=0$.

It remains to justify the remainder without differentiating an undifferentiated
big-$O$ term.  The first term of \eqref{eq:modified-kupradze} is not differentiated.
After the quadratic contribution has been retained, Taylor's theorem gives directly
\begin{equation}
\label{eq:first-yukawa-remainder}
 \left|
 \frac1\mu\Phi_{\kappa_{\mathrm S}}(x)I_3
 -
 \left[
 \frac{I_3}{4\pi\mu\varrho}
 -\frac{\kappa_{\mathrm S}}{4\pi\mu}I_3
 +\frac{\kappa_{\mathrm S}^2\varrho}{8\pi\mu}I_3
 \right]
 \right|
 \le
 C_Rs^3\varrho^2.
\end{equation}
For the Hessian term, define
\[
 R_\kappa(\varrho)
 :=
 \kappa^5\varrho^4\mathfrak r_5(\kappa\varrho).
\]
By \eqref{eq:r5-derivatives}, for $j=0,1,2$ and $0<\varrho\le2R$,
\begin{equation}
\label{eq:radial-remainder-derivatives}
 |R_\kappa^{(j)}(\varrho)|
 \le
 C_R\kappa^5\varrho^{4-j}.
\end{equation}
For a radial scalar function $f(\varrho)$,
\begin{equation*}
 \nabla\nabla^\top f(\varrho)
 =
 f''(\varrho)\frac{xx^\top}{\varrho^2}
 +
 \frac{f'(\varrho)}\varrho
 \left(
 I_3-\frac{xx^\top}{\varrho^2}
 \right).
\end{equation*}
Hence \eqref{eq:radial-remainder-derivatives} implies
\begin{equation*}
 \left|
 \nabla\nabla^\top R_\kappa(\varrho)
 \right|
 \le
 C_R\kappa^5\varrho^2.
\end{equation*}
Apply this estimate to $\kappa_{\mathrm S}$ and $\kappa_{\mathrm P}$ and take the difference.
Multiplication by $(\rho_0s^2)^{-1}$ gives an $O(s^3\varrho^2)$ remainder for the Hessian
part.  Together with
\eqref{eq:first-yukawa-remainder}, this proves
\[
 |\bfE_s(x)|\le C_Rs^3\varrho^2,
 \qquad
 0<\varrho\le2R.
\]
Finally,
\[
 \|\bfE_s*\bfh\|_{L^\infty(B_R)}
 \le
 C_Rs^3
 \sup_{x\in B_R}\int_Q|x-y|^2|\bfh(y)|\,dy
 \le
 C_Rs^3\|\bfh\|_{L^1(Q)},
\]
and the corresponding bounds for the displayed lower-order kernels give
\eqref{eq:operator-expansion}.  This proves Lemma~\ref{lem:kernel-expansion}.

\section{Equivalence of the variational and integral formulations}
\label{app:variational}

We verify that the fixed-point solution used in the low-frequency expansion coincides with the
Bochner--Laplace transform of the energy solution.

Fix $s>0$.  Let $L_s:H^1(\R^3)^3\to H^{-1}(\R^3)^3$ denote the operator induced by the
coercive form
\[
 a(\bfw,\bfv)+\rho_0s^2(\bfw,\bfv)_{L^2};
\]
in distributional notation, $L_s=\cA_{\lambda,\mu}+\rho_0s^2$.  Lax--Milgram gives an
isomorphism
\[
 L_s:H^1(\R^3)^3\longrightarrow H^{-1}(\R^3)^3.
\]
For $\bfh\in L^2(\R^3)^3$, the Fourier multiplier formula
\eqref{eq:Gamma-symbol} also gives the elliptic estimate
\begin{equation}
\label{eq:Ls-H2}
 \|L_s^{-1}\bfh\|_{H^2(\R^3)}
 \le
 C_s\|\bfh\|_{L^2(\R^3)}.
\end{equation}
The two multiplier eigenvalues are
$(\mu|\xi|^2+\rho_0s^2)^{-1}$ and
$((\lambda+2\mu)|\xi|^2+\rho_0s^2)^{-1}$.  For fixed $s>0$, both
\[
 \frac{1+|\xi|^2}{\mu|\xi|^2+\rho_0s^2}
 \quad\text{and}\quad
 \frac{1+|\xi|^2}{(\lambda+2\mu)|\xi|^2+\rho_0s^2}
\]
are bounded functions of $\xi$.  On compactly supported $L^2$ data,
$L_s^{-1}\bfh=\bfGamma_s*\bfh$.

Let $\bfv\in L^\infty(B_R)^3$ be a fixed point of the contraction equation.  Since
$q$ is bounded with $\esssupp q\subset\overline Q\subset B_R$,
\[
 \bfh_{\bfv}
 :=
 \bfG+s\bfF-s^2q\bfv
 \in L^2(\R^3)^3,
 \qquad
 \esssupp\bfh_{\bfv}\subset\overline Q.
\]
Here and below, $q\bfv$ denotes the function on $\R^3$ equal to $q(x)\bfv(x)$ on $Q$
and zero outside $Q$; thus no extension of $\bfv$ away from $Q$ is used.
Define globally
\[
 \bfU:=\bfGamma_s*\bfh_{\bfv}.
\]
By \eqref{eq:Ls-H2}, $\bfU\in H^2(\R^3)^3\subset H^1(\R^3)^3$, and by local Sobolev
embedding in three dimensions $\bfU\in L^\infty_{\mathrm{loc}}$.  The fixed-point identity says
that $\bfU=\bfv$ on $B_R$, in particular on $Q$.  Hence
\[
 L_s\bfU
 =
 \bfG+s\bfF-s^2q\bfU
\]
globally, so $\bfU$ satisfies the variational equation \eqref{eq:laplace-variational}.

Conversely, let $\widetilde{\bfu}\in H^1(\R^3)^3$ be the unique solution of \eqref{eq:laplace-variational}.
Because $q\in L^\infty(\R^3)$, $\esssupp q\subset\overline Q$, and
$\widetilde{\bfu}\in L^2(\R^3)^3$,
\[
 \bfh
 :=
 \bfG+s\bfF-s^2q\widetilde{\bfu}
 \in L^2(\R^3)^3,
 \qquad
 \esssupp\bfh\subset\overline Q.
\]
Thus
\[
 \widetilde{\bfu}
 =
 L_s^{-1}\bfh
 =
 \bfGamma_s*\bfh
 \in H^2(\R^3)^3,
\]
and in particular its restriction to $B_R$ belongs to $L^\infty(B_R)^3$ and satisfies the
same fixed-point equation.  For sufficiently small $s$, that fixed point is unique by
Proposition~\ref{prop:field-expansion}.  Hence, for sufficiently small $s$, the variational
solution, the Lippmann--Schwinger fixed point, and the Bochner--Laplace transform of the
time-domain energy solution all coincide.

\section*{Declaration on the use of generative AI}
Generative AI tools were used for language editing, organization, and mechanical
consistency checks.  The authors independently verified all mathematical arguments,
calculations, and references and take full responsibility for the final manuscript.

\section*{Acknowledgments}
The research of Yixian Gao was supported by the NSFC (12371187) and the STDPP of
Jilin Province (20240101006JJ).  The research of Hongyu Liu was supported by the
Hong Kong Research Grants Council General Research Fund (g11311122, 11304224,
and 11300821), the NSFC/RGC Joint Research Scheme (N\_CityU101/21), and the ANR/RGC
Joint Research Scheme (A\_CityU203/19).  The research of Yang Liu was supported by the
NSFC (12401554) and the STDPP of Jilin Province (20260102252JC).


\begin{thebibliography}{99}

\bibitem{BaoHuKianYin2018}
G.~Bao, G.~Hu, Y.~Kian, and T.~Yin,
\emph{Inverse source problems in elastodynamics},
Inverse Problems \textbf{34} (2018), 045009;
\href{https://doi.org/10.1088/1361-6420/aaaf7e}{doi:10.1088/1361-6420/aaaf7e}.

\bibitem{BellassouedImanuvilovYamamoto2008}
M.~Bellassoued, O.~Imanuvilov, and M.~Yamamoto,
\emph{Inverse problem of determining the density and two Lam\'e coefficients by boundary data},
SIAM J. Math. Anal. \textbf{40} (2008), 238--265;
\href{https://doi.org/10.1137/070679971}{doi:10.1137/070679971}.

\bibitem{Bhattacharyya2018}
S.~Bhattacharyya,
\emph{Local uniqueness of the density from partial boundary data for isotropic elastodynamics},
Inverse Problems \textbf{34} (2018), 125001;
\href{https://doi.org/10.1088/1361-6420/aade76}{doi:10.1088/1361-6420/aade76}.

\bibitem{BhattacharyyaDeHoopKatsnelsonUhlmann2022SIIMS}
S.~Bhattacharyya, M.~V.~de Hoop, V.~Katsnelson, and G.~Uhlmann,
\emph{Recovery of piecewise smooth density and Lam\'e parameters from high-frequency exterior Cauchy data},
SIAM J. Imaging Sci. \textbf{15} (2022), 1910--1943;
\href{https://doi.org/10.1137/22M1480951}{doi:10.1137/22M1480951}.

\bibitem{BlastenLin2019}
E.~Bl{\aa}sten and Y.-H.~Lin,
\emph{Radiating and non-radiating sources in elasticity},
Inverse Problems \textbf{35} (2019), 015005;
\href{https://doi.org/10.1088/1361-6420/aae99e}{doi:10.1088/1361-6420/aae99e}.

\bibitem{ChenJiangLiuLoTao2026}
L.~Chen, Y.~Jiang, H.~Liu, C.~W.~K.~Lo, and L.~Tao,
\emph{Determining evolutionary equations from a single passive boundary observation},
preprint, arXiv:2505.08473, 2026.

\bibitem{DeHoopNakamuraZhai2017}
M.~V.~de Hoop, G.~Nakamura, and J.~Zhai,
\emph{Reconstruction of Lam\'e moduli and density at the boundary enabling directional elastic wavefield decomposition},
SIAM J. Appl. Math. \textbf{77} (2017), 520--536;
\href{https://doi.org/10.1137/16M1073406}{doi:10.1137/16M1073406}.

\bibitem{DeHoopNakamuraZhai2019}
M.~V.~de Hoop, G.~Nakamura, and J.~Zhai,
\emph{Unique recovery of piecewise analytic density and stiffness tensor from the elastic-wave Dirichlet-to-Neumann map},
SIAM J. Appl. Math. \textbf{79} (2019), 2359--2384;
\href{https://doi.org/10.1137/18M1232802}{doi:10.1137/18M1232802}.

\bibitem{Feizmohammadi2025}
A.~Feizmohammadi,
\emph{Reconstruction of 1D evolution equations and their initial data from one passive measurement},
SIAM J. Math. Anal. \textbf{57} (2025), 5089--5106;
\href{https://doi.org/10.1137/25M1728533}{doi:10.1137/25M1728533}.

\bibitem{GaoLiuLiu2023}
Y.~Gao, H.~Liu, and Y.~Liu,
\emph{On an inverse problem for the plate equation with passive measurement},
SIAM J. Appl. Math. \textbf{83} (2023), 1196--1214;
\href{https://doi.org/10.1137/22M1499881}{doi:10.1137/22M1499881}.

\bibitem{GohbergLancasterRodman2009}
I.~Gohberg, P.~Lancaster, and L.~Rodman,
\emph{Matrix Polynomials},
Classics in Applied Mathematics 58, Society for Industrial and Applied Mathematics,
Philadelphia, 2009;
\href{https://doi.org/10.1137/1.9780898719024}{doi:10.1137/1.9780898719024}.

\bibitem{HuKian2020}
G.~Hu and Y.~Kian,
\emph{Uniqueness and stability for the recovery of a time-dependent source in elastodynamics},
Inverse Probl. Imaging \textbf{14} (2020), 463--487;
\href{https://doi.org/10.3934/ipi.2020022}{doi:10.3934/ipi.2020022}.

\bibitem{ImanuvilovYamamoto2005}
O.~Y.~Imanuvilov and M.~Yamamoto,
\emph{Carleman estimates for the non-stationary Lam\'e system and the application to an inverse problem},
ESAIM Control Optim. Calc. Var. \textbf{11} (2005), 1--56;
\href{https://doi.org/10.1051/cocv:2004030}{doi:10.1051/cocv:2004030}.

\bibitem{KianLiu2025}
Y.~Kian and H.~Liu,
\emph{Uniqueness and stability in determining the wave equation from a single passive boundary measurement},
preprint, arXiv:2507.10012, 2025.

\bibitem{KianUhlmann2025}
Y.~Kian and G.~Uhlmann,
\emph{Determination of the sound speed and an initial source in photoacoustic tomography},
Trans. Amer. Math. Soc. \textbf{378} (2025), 5329--5353;
\href{https://doi.org/10.1090/tran/9467}{doi:10.1090/tran/9467}.

\bibitem{KnoxMoradifam2020}
C.~Knox and A.~Moradifam,
\emph{Determining both the source of a wave and its speed in a medium from boundary measurements},
Inverse Problems \textbf{36} (2020), 025002;
\href{https://doi.org/10.1088/1361-6420/ab53fc}{doi:10.1088/1361-6420/ab53fc}.

\bibitem{LiuUhlmann2015}
H.~Liu and G.~Uhlmann,
\emph{Determining both sound speed and internal source in thermo- and photo-acoustic tomography},
Inverse Problems \textbf{31} (2015), 105005;
\href{https://doi.org/10.1088/0266-5611/31/10/105005}{doi:10.1088/0266-5611/31/10/105005}.

\bibitem{McLean2000}
W.~McLean,
\emph{Strongly Elliptic Systems and Boundary Integral Equations},
Cambridge University Press, Cambridge, 2000;
\href{https://doi.org/10.1017/CBO9780511546335}{doi:10.1017/CBO9780511546335}.

\bibitem{Moradifam2026}
A.~Moradifam,
\emph{Simultaneous recovery of the initial source and sound speed for the wave equation under a constitutive constraint},
preprint, arXiv:2607.19799, 2026.

\bibitem{Rachele2003}
L.~V.~Rachele,
\emph{Uniqueness of the density in an inverse problem for isotropic elastodynamics},
Trans. Amer. Math. Soc. \textbf{355} (2003), 4781--4806;
\href{https://doi.org/10.1090/S0002-9947-03-03268-9}{doi:10.1090/S0002-9947-03-03268-9}.

\bibitem{StefanovUhlmann2013}
P.~Stefanov and G.~Uhlmann,
\emph{Recovery of a source term or a speed with one measurement and applications},
Trans. Amer. Math. Soc. \textbf{365} (2013), 5737--5758;
\href{https://doi.org/10.1090/S0002-9947-2013-05703-0}{doi:10.1090/S0002-9947-2013-05703-0}.

\bibitem{Tittelfitz2012}
J.~Tittelfitz,
\emph{Thermoacoustic tomography in elastic media},
Inverse Problems \textbf{28} (2012), 055004;
\href{https://doi.org/10.1088/0266-5611/28/5/055004}{doi:10.1088/0266-5611/28/5/055004}.

\bibitem{UhlmannZhai2024}
G.~Uhlmann and J.~Zhai,
\emph{Determination of the density in a nonlinear elastic wave equation},
Math. Ann. \textbf{390} (2024), 2825--2858;
\href{https://doi.org/10.1007/s00208-024-02797-w}{doi:10.1007/s00208-024-02797-w}.

\bibitem{Zhai2026}
J.~Zhai,
\emph{Determination of the density in the linear elastic wave equation},
J. Differential Equations \textbf{459} (2026), 114064;
\href{https://doi.org/10.1016/j.jde.2025.114064}{doi:10.1016/j.jde.2025.114064}.

\end{thebibliography}
\end{document}